\documentclass[11pt]{article}
\usepackage[margin=1.0in]{geometry}
\usepackage{graphicx}
\usepackage{enumitem}
\usepackage{amsmath,amssymb,amsfonts,amsthm}
\usepackage{algorithm,algpseudocode}
\usepackage{array,longtable,booktabs}
\usepackage{url}
\usepackage{hyperref}
\usepackage{fancyhdr}
\usepackage{threeparttable}
\usepackage{multirow}
\usepackage{makecell}

\usepackage{microtype}
\newtheorem{theorem}{Theorem}
\newtheorem{lemma}[theorem]{Lemma}
\newtheorem{proposition}[theorem]{Proposition}

\theoremstyle{definition}

\theoremstyle{remark}

\numberwithin{equation}{section}

\allowdisplaybreaks

\setlist[enumerate,1]{label=(\roman*),font=\normalfont}

\title{The fixed-point bundle method over product-of-simplex domains arising from game equilibria}

\author{
  Hongbo Sun\\
  \texttt{shb20@tsinghua.org.cn}
}

\begin{document}
\date{}
\maketitle

\begin{abstract}
  Developing general-purpose algorithms with global asymptotic convergence guarantees for Nash equilibria of general-sum games remains a longstanding challenge. This paper extends the fixed-point bundle framework for finite-dimensional variational inequalities (VIs) from simplex domains to product-of-simplex domains, enabling its direct application to Nash equilibrium computation. The resulting fixed-point bundle exhibits a composite fiber bundle structure. The key innovation is to construct an equivalent VI on a simplex and establish the equivalence between the two fixed-point bundle frameworks via a fiber bundle isomorphism. Exploiting this geometric equivalence, we show that the path-following algorithm for VIs on product-of-simplex domains inherits the global linear convergence guarantees of the simplex-domain framework. In numerical experiments, the algorithm attained a precision of $10^{-5}$ on all 5600 randomly generated game instances, with dimension ranging from 2-player 128-action to 128-player 2-action. Code and data are available at \url{https://github.com/shb20tsinghua/FiberBundle_VI}.
\end{abstract}



\tableofcontents

\section{Introduction}

Game equilibrium problems, introduced by Nash \cite{nash1951non}, are central in economics \cite{kreps1990game}, engineering \cite{han2012game}, artificial intelligence \cite{hazra2022applications}, and other fields. They describe noncooperative settings in which each agent selects a strategy to maximize its own utility given the strategies of others, and an equilibrium is a strategy profile from which no agent has a unilateral incentive to deviate. Under some regularity conditions, a game equilibrium problem can be equivalently reformulated as a variational inequality (VI) \cite{dual_gap}. Let $K\subseteq\mathbb{R}^n$ be a nonempty compact convex set and $H:K\to\mathbb{R}^n$ be a continuous function. The VI problem ${\rm VI}(H,K)$ asks for $x^\star\in K$ such that
$$\left\langle H(x^\star),x'-x^\star\right\rangle \geq0\quad {\rm for~all~} x'\in K.$$
${\rm VI}(H,K)$ is naturally described by the gap function
$${\rm gap}(x)=\sup_{x'\in K}\left\langle H(x),x-x'\right\rangle,$$
where ${\rm gap}(x)\geq0$ for all $x\in K$, and $x^\star\in K$ is a solution if and only if ${\rm gap}(x^\star)=0$. When $K$ has a product structure such that $K=\prod_iK_i$, and $H$ is a block-wise gradient field such that $H=[H_i]=[\nabla_{x_i} f_i]$, the problem ${\rm VI}(H,K)$ is a game equilibrium problem, with $K_i$ the strategy set of each player and $f_i$ the payoff function of each player.

In terms of iterative methods with asymptotic convergence, existing algorithms for game equilibrium problems are largely the same as those for VIs \cite{facchinei2007generalized}. VI algorithms fall into several categories, each with its application to game equilibrium problems. Projection \cite{belgioioso2018projected} and proximal point methods \cite{yi2018distributed} reformulate the equilibrium conditions as a fixed-point equation and iterate projections onto the feasible set. Merit-function approaches \cite{dreves2013globalized} convert the game into an optimization problem via merit functions and apply gradient or Newton updates. Interior-point methods \cite{dreves2011solution} transform the game into a complementarity problem and follow a central path subject to the perturbed KKT system. Homotopy methods \cite{herings2010homotopy} deform a simple operator into the target game operator. Operator splitting methods \cite{yi2019operator} exploit decompositions such as sums of individual monotone operators. Modern variants add adaptive step sizes, inertial terms, or stochastic estimates for large-scale strategic interactions \cite{vinh2019inertial,franci2022stochastic}.

Besides VI algorithms, there are also machine-learning algorithms that operate at the level of individual players, such as self-play dynamics \cite{fp,fsp} that update each player's strategy with the best-response strategy, and no-regret dynamics \cite{noregret,cfr} that update each player's strategy to decrease the regret relative to each pure action.

Developing general-purpose and efficient algorithms for Nash equilibria remains a longstanding challenge. First, computing a Nash equilibrium of a general-sum normal-form game is PPAD-complete \cite{ppad}, so no polynomial-time algorithm is expected unless PPAD collapses. Moving from polynomial-time guarantees to a broader asymptotic convergence perspective, existing VI algorithms lack global convergence guarantees without additional structural assumptions, and machine-learning algorithms do not guarantee convergence to a Nash equilibrium either. In short, in terms of iterative methods with asymptotic convergence, no existing method guarantees global convergence to a Nash equilibrium of an arbitrary general-sum normal-form game.

This paper mainly deals with VIs with the product-of-simplex domain, which is the canonical domain for mixed strategies of normal-form games. We extend the fixed-point bundle framework in Paper 1 \cite{mine} for VIs on simplex domains to VIs on product-of-simplex domains. In Paper 1, the fixed-point bundle framework formulates the solution space of a VI on a simplex domain as a fiber bundle, in which starting point selection, path-following, and singularity avoidance are systematically integrated. The path-following algorithm applies to VIs with general continuous functions and general polytope domains after preprocessing, and it guarantees global linear convergence for nonsingular solutions and retains global linear convergence up to a prescribed precision for singular solutions. This paper establishes an equivalence between the fixed-point bundle framework over product-of-simplex domains and that over simplex domains, showing that the path-following algorithm over product-of-simplex domains has the same convergence result as in Paper 1.

The proposed approach begins by preprocessing the general ${\rm VI}(H,K)$ into ${\rm VI}(F,\Delta_n^m)$ with a real-analytic operator $F$ and domain $\Delta_n^m$ (the product of $m$ $n$-dimensional simplices), in correspondence with the preprocessing in Paper 1.

The three equivalent characterizations in Paper 1 can be generalized to ${\rm VI}(F,\Delta_n^m)$ via the decomposable structure of the gap function of ${\rm VI}(F,\Delta_n^m)$. We then derive the fixed-point bundle of ${\rm VI}(F,\Delta_n^m)$, which turns out to have a composite structure on top of the ordinary fiber bundle structure.

Based on the homeomorphism $\phi:\Delta_n^m\times\Delta_m^\circ\to\Delta_{mn}^+$, an operator $\bar{F}$ can be constructed such that ${\rm VI}(\bar{F},\Delta_{mn}^+)$ and ${\rm VI}(F,\Delta_n^m)$ have the same solutions. The central innovation of this paper is to establish the equivalence between the fixed-point bundle frameworks of ${\rm VI}(F,\Delta_n^m)$ and ${\rm VI}(\bar{F},\Delta_{mn}^+)$ via a fiber bundle isomorphism. The fixed-point bundle framework can be defined for ${\rm VI}(\bar{F},\Delta_{mn}^+)$, but the theoretical guarantees in Paper 1 require an additional extension argument to apply, since ${\rm VI}(\bar{F},\Delta_{mn}^+)$ is not a smooth VI on the closed simplex.

The predictor-corrector path-following algorithm for ${\rm VI}(F,\Delta_n^m)$ and ${\rm VI}(\bar{F},\Delta_{mn}^+)$ turns out to fix the variable on a specific slice given by $w_{\rm init}\in\Delta_m^\circ$. The restriction of the map $\bar{F}$ to this slice can be extended to a map $\tilde{F}$ that is real-analytic on the closed simplex $\Delta_{mn}$, so that all the results in Paper 1 apply to ${\rm VI}(\tilde{F},\Delta_{mn})$. Then, transferring from ${\rm VI}(\tilde{F},\Delta_{mn})$ to ${\rm VI}(\bar{F},\Delta_{mn}^+)$ and then to ${\rm VI}(F,\Delta_n^m)$, the path-following algorithm for ${\rm VI}(F,\Delta_n^m)$ has the same convergence guarantee as in Paper 1.

In experiments, numerical results on 5600 randomly generated instances, with dimensions ranging from 2-player 128-action to 128-player 2-action, show that the algorithm reduces the gap function below $10^{-5}$ in every case.

The remainder of the paper is organized as follows. Section 2 details the preprocessing. Section 3 generalizes the three equivalent characterizations and the fixed-point bundle. Section 4 constructs a VI on a simplex and establishes its equivalence to the VI on the product-of-simplex. Section 5 transforms the predictor-corrector scheme and transfers convergence. Section 6 presents experimental results. Concluding remarks appear in Section 7.

For vectors, $a\circ b$ denotes componentwise multiplication and $a/b$ componentwise division, and the inequality $a\geq b$ is also interpreted componentwise. For a matrix $A$, componentwise multiplication is interpreted as $A\circ b=A{\rm diag}(b)$ and $b\circ A={\rm diag}(b)A$. For tensors, $\mathbb{R}^{mn}$ is a vector space of dimension $mn$, $\mathbb{R}^{m\times n}$ is a tensor space of shape $m\times n$, and the operation ${\rm reshape}(\cdot)$ always represents switching between these two shapes. For block stacking, $[a_i]$ represents the concatenation of small vectors $a_i$ into a large vector, and ${\rm Diag}[x_iI]$ represents a block diagonal matrix. The operation ${\rm norm}([x_i])=[x_i]/(\sum_j x_j)$ denotes normalization. The operation $w\odot\pi$ is defined for $w\in\mathbb{R}^m$ and $\pi\in\mathbb{R}^{m\times n}$ by $w\odot\pi\in\mathbb{R}^{m\times n}$ with $(w\odot\pi)_i=w_i\pi_i$. Finally, $\mathbf{1}$ is the all-one vector; it sometimes has dimension $n$ and sometimes dimension $mn$, but its dimension can always be inferred from the context.

\section{Preprocessing: smooth VI over product-of-simplex}

As in Paper 1, we first transform the general ${\rm VI}(H,K)$ with a continuous map and polytope domain into ${\rm VI}(F,\Delta_n^m)$ with a smooth map and a simplex product domain. The subsequent theory is then based on ${\rm VI}(F,\Delta_n^m)$.

In Paper 1, the fixed-point bundle method applies to VIs over polytope domains, where polytopes and simplices are in one-to-one correspondence via linear transformations when they have the same number of vertices. The linear transformation $x=\left[v_1,\dots,v_n\right]\sigma$ transforms a simplex $\Delta$ to a polytope with vertices $v_1,\dots,v_n$. In this paper, a linear transformation
$$x=V{\rm reshape}(\pi)=\sum_{i=1}^m V_i\pi_i=\sum_{i=1}^m \left[v_{i1},\dots,v_{in}\right]\pi_i$$
transforms a simplex product $\Delta_n^m$ into a Minkowski sum of $m$ polytopes, where each summand has vertices $v_{i1},\dots,v_{in}$.

The Minkowski sum of polytopes is still a polytope, and if each summand polytope has $n_i$ vertices, the resulting polytope can have at most $\prod_in_i$ vertices. For instance, a polytope $K$ with a product structure $K=K_1\times\dots\times K_m$ is a Minkowski sum in a higher-dimensional space and has $\prod_in_i$ vertices. Thus, representing a polytope as a linear transformation of a simplex product can reduce the input vertex budget from $\prod_in_i$ to $\sum_in_i$ compared with the plain simplex representation, which may yield a substantial compression and result in a dimension reduction for the fixed-point bundle method.

The linear transformation $x=V{\rm reshape}(\pi)$ transforms the VI as follows.
$$\left\langle H(x),x'-x\right\rangle=\sum_i\left\langle H_i(\textstyle\sum_i V_i\pi_i),V_i\pi_i^\prime-V_i\pi_i\right\rangle=\sum_i\left\langle V_i^\top H_i(\textstyle\sum_i V_i\pi_i),\pi_i^\prime-\pi_i\right\rangle\geq0.$$
As in Paper 1, we then smooth the continuous mapping $F(\pi)={\rm reshape}(V^\top H(V{\rm reshape}(\pi)))$ using Dirichlet distribution averaging. The Dirichlet distribution ${\rm Dir}(\alpha)$ on the simplex $\Delta$ is given by
\begin{equation}
  {\rm Dir}(y;\alpha)=\frac{1}{B(\alpha)}\prod_{i=1}^n y_i^{\alpha_i-1},\quad\alpha>0,
\end{equation}
where $y\in\Delta$ is the variable, $\alpha>0$ is a vector of concentration parameters, and $B(\alpha)$ is the Beta function. We need a ${\rm Dir}(\alpha_i)$ for each $\pi_i\in\Delta_n$, so the map $F(\pi)$ is actually smoothed with the product Dirichlet distribution $\prod_i{\rm Dir}(\alpha_i)$. For each component $F_{i,a}$ of $F$, we define the following operation.
\begin{equation}
  \begin{aligned}
    (T_j F_{i,a})(\pi_1,\dots,\pi_m)= & \mathbb{E}_{Y_j\sim{\rm Dir}(\alpha_j)}\left[F_{i,a}(\pi_1,\dots,Y_j,\dots,\pi_m)\right] \\
    =                                 & \int_{\Delta_n} F_{i,a}(\pi_1,\dots,y_j,\dots,\pi_m){\rm Dir}(y_j;\alpha_j)\,dy_j        \\
  \end{aligned}
\end{equation}

Then, the product Dirichlet distribution smoothing of a continuous map $F:\Delta\to\mathbb{R}^n$ is the map $F_\epsilon:\Delta\to\mathbb{R}^n$ given componentwise by
\begin{equation}
  \label{equ_mollifi}
  F_{\epsilon,i,a}(\pi)=\mathbb{E}_{Y\sim\prod_{j=1}^m{\rm Dir}(\pi_j/\epsilon+c_j)}\left[F_{i,a}(Y)\right]=(T_1\circ\dots\circ T_m F_{i,a})(\pi),
\end{equation}
where $\epsilon>0$ is a scalar that measures the approximation error, and $c_i>0$ is a constant vector ensuring that $\pi_i/\epsilon+c_i>0$ for all $\pi_i\in\Delta_n$.

Following Paper 1, we then prove that $F_\epsilon(\pi)$ is real-analytic on $\Delta_n^m$. Instead of proving this from scratch, we reuse the analyticity result for single-variable Dirichlet distribution averaging in Paper 1. Hartogs' theorem on separate holomorphy \cite{krantz2001function} from complex analysis can then be used to extend this result to analyticity for multiple-variable Dirichlet distribution averaging. This theorem states that for a function $g:U\times V\to\mathbb{C}$ on $U\subset\mathbb{C}^m$ and $V\subset\mathbb{C}^n$, if for every fixed $u\in U$, the map $v\mapsto g(u,v)$ is holomorphic on $U$, and for every fixed $v\in V$, the map $u\mapsto g(u,v)$ is holomorphic on $V$, then $g$ is holomorphic on $U\times V$.

\begin{proposition}
  \label{thm_analytic}
  For every $\epsilon>0$, $F_\epsilon(\pi)$ is real-analytic on $\Delta_n^m$.
\end{proposition}
\begin{proof}
  In the proof of Proposition 1, it is shown that if $f(\sigma)$ is continuous on the simplex $\Delta$, then for every $\sigma\in\Delta$, $(Tf)(z)$ is holomorphic on the complex neighborhood $U_\sigma=\{z\in\mathbb{C}^n|\lVert z-\sigma\rVert\leq d\}$.

  In our case, for every $\pi\in\Delta_n^m$, we study $(T_1\circ\dots\circ T_m F_{i,a})(z_1,\dots,z_m)$ on the product complex neighborhood $U_\pi=\prod_j U_{\pi_j}$ with $U_{\pi_j}=\{z_j\in\mathbb{C}^n|\lVert z_j-\pi_j\rVert\leq d_j\}$. We denote $(y_k^\prime,y_{\hat{k}}):=(y_1,\dots,y_{k-1},y_k^\prime,y_{k+1},\dots,y_m)$ and $T_{\hat{k}}:=T_1\circ\dots\circ T_{k-1}\circ T_{k+1}\circ\dots\circ T_m$.

  First, the operators $T_j$ are commutative. The operation $(T_1\circ\dots\circ T_m F_{i,a})(z_1,\dots,z_m)$ is an integral of the continuous function $F_{i,a}(y_1,\dots,y_m)$ over the compact region $\Delta_n^m$ against the product measure $\prod_{j=1}^m{\rm Dir}(y_j;z_j/\epsilon+c_j)dy_j$. By the Fubini-Tonelli theorem \cite{analytic}, if a multiple integral integrates a continuous integrand over a compact region, the order of integration can be interchanged. Thus, for every index $k$, we can write $F_{\epsilon,i,a}=T_k\circ T_{\hat{k}}F_{i,a}$.

  Next, we prove separate holomorphy for every index $k$. Since $F_{i,a}(y)$ is continuous on the compact set $\Delta_n^m$, it is uniformly continuous: for any $\epsilon>0$ there exists $\delta>0$ such that whenever $\lVert y_k-y_k^\prime\rVert<\delta$, we have $\lvert F_{i,a}(y_k^\prime,y_{\hat{k}})-F_{i,a}(y_k,y_{\hat{k}})\rvert<\epsilon$. Then, for fixed $z_j\in U_{\pi_j}$ with $j\neq k$, we have
  \begin{align*}
    \lvert (T_{\hat{k}}F_{i,a})(y_k,z_{\hat{k}})-(T_{\hat{k}}F_{i,a})(y_k^\prime,z_{\hat{k}})\rvert\leq & \int_{\Delta_n^{m-1}} \lvert F_{i,a}(y_k^\prime,y_{\hat{k}})-F_{i,a}(y_k,y_{\hat{k}})\rvert\prod_{j\neq k}{\rm Dir}(y_j;z_j/\epsilon+c_j)\,dy_j \\
    <                                                                                                   & \epsilon\int_{\Delta_n^{m-1}}\prod_{j\neq k}{\rm Dir}(y_j;z_j/\epsilon+c_j)\,dy_j=\epsilon
  \end{align*}
  Thus, $(T_{\hat{k}}F_{i,a})(y_k,z_{\hat{k}})$ is (uniformly) continuous in $y_k$ on $\Delta$. Then, the analyticity result for single-variable Dirichlet distribution averaging in the proof of Proposition 1 of Paper 1 applies to $F_{\epsilon,i,a}=T_k(T_{\hat{k}}F_{i,a})$, and we obtain that for every fixed $z_j\in U_{\pi_j}$ with $j\neq k$, the map $z_k\mapsto F_{\epsilon,i,a}(z_k,z_{\hat{k}})$ is holomorphic in $z_k$ on $U_{\pi_k}$.

  Then, Hartogs' theorem on separate holomorphy applies, and we obtain that $F_{\epsilon,i,a}(z)$ is holomorphic in $z$ on $U_\pi$. Since $\pi\in\Delta_n^m$ is arbitrary, it follows that $F_{\epsilon,i,a}(\pi)$ is real-analytic in $\pi$ on $\Delta_n^m$. The vector function $F_\epsilon(\pi)$ is real-analytic in $\pi$ on $\Delta_n^m$ because each of its components indexed by $(i,a)$ is real-analytic.
\end{proof}

\section{The fixed-point bundle of VI on simplex product}

\subsection{Generalizing the three equivalent characterizations}

As in Paper 1, we deal with ${\rm VI}(F,\Delta_n^m)$ assuming that $F$ is real-analytic on $\Delta_n^m$ throughout the following developments. In Paper 1, the analysis is built on the three equivalent characterizations of ${\rm VI}(\bar{F},\Delta)$. We therefore begin by transforming these three characterizations so as to characterize ${\rm VI}(F,\Delta_n^m)$. First, we show that ${\rm VI}(F,\Delta_n^m)$ can be decomposed into several VIs on $\Delta_n$ via its gap function.
\begin{theorem}
  \label{thm_decompose}
  For every $\pi\in\Delta_n^m$, we have
  \begin{equation}
    {\rm gap}(\pi)=\sum_i{\rm gap}_i(\pi_i,\pi_{\hat{i}}),\quad{\rm where~}{\rm gap}_i(\pi_i,\pi_{\hat{i}}):=\sup_{\pi_i^\prime\in\Delta_n}\left\langle F_i(\pi_i,\pi_{\hat{i}}),\pi_i-\pi_i^\prime\right\rangle
  \end{equation}
  Consequently, for every $\pi\in\Delta_n^m$, ${\rm gap}(\pi)=0$ if and only if ${\rm gap}_i(\pi_i,\pi_{\hat{i}})=0$ for every index $i$.
\end{theorem}
\begin{proof}
  First, we derive ${\rm gap}(\pi)$ as follows.
  \begin{align*}
    {\rm gap}(\pi)= & \sup_{\pi'\in\Delta_n^m}\left\langle F(\pi),\pi-\pi'\right\rangle=\sum_i\left(\pi_i^\top F_i(\pi)-\min_a F_{i,a}(\pi)\right)     \\
    =               & \sum_i\sup_{\pi_i^\prime\in\Delta_n}\left\langle F_i(\pi),\pi_i-\pi_i^\prime\right\rangle=\sum_i{\rm gap}_i(\pi_i,\pi_{\hat{i}})
  \end{align*}
  Thus, we obtain ${\rm gap}(\pi)=\sum_i{\rm gap}_i(\pi_i,\pi_{\hat{i}})$.

  For every $\pi\in\Delta_n^m$, we have ${\rm gap}(\pi)\geq0$ and ${\rm gap}_i(\pi_i,\pi_{\hat{i}})\geq0$ by definition. Then, by ${\rm gap}(\pi)=\sum_i{\rm gap}_i(\pi_i,\pi_{\hat{i}})$, we have ${\rm gap}(\pi)=0$ if and only if ${\rm gap}_i(\pi_i,\pi_{\hat{i}})=0$ for every index $i$.
\end{proof}

In Theorem~\ref{thm_decompose}, ${\rm gap}_i(\pi_i,\pi_{\hat{i}})$ can be viewed as the gap function of a parametric variational inequality ${\rm VI}(F_i(\cdot,\pi_{\hat{i}}),\Delta_n)$ with $\pi_{\hat{i}}$ as the parameter. The theorem implies that $\pi$ is a solution of ${\rm VI}(F,\Delta_n^m)$ if and only if $\pi_i$ is a solution of ${\rm VI}(F_i(\cdot,\pi_{\hat{i}}),\Delta_n)$ for each index $i$.

Since ${\rm VI}(F_i(\cdot,\pi_{\hat{i}}),\Delta_n)$ is a VI on a simplex, the three equivalent characterizations in Paper 1 apply to it. However, they are compatible over the index $i$ only pointwise for individual $\pi\in\Delta_n^m$, at which point all the parameters $\pi_{\hat{i}}$ and variables $\pi_i$ across the indices $i$ come from the same $\pi$. In Paper 1, the three equivalent characterizations are indeed pointwise: their definitions consist only of functions, equalities, and inequalities that are essentially pointwise objects assigning values to each individual $(\pi_i,\tau_i)$, and the equivalence is also in pointwise form, stating that for each individual $(\pi_i,\tau_i)$, the properties of satisfying the three characterizations are equivalent. Thus, stacking the three equivalent characterizations of ${\rm VI}(F_i(\cdot,\pi_{\hat{i}}),\Delta_n)$ over the index $i$ at each individual $\pi\in\Delta_n^m$, we obtain the generalization characterizing ${\rm VI}(F,\Delta_n^m)$. We briefly go through them, since the generalization merely adds a subscript $i$ to those in Paper 1.

\noindent\textbf{Linear programming form.}
Equation~\eqref{equ_fpkkt} is the fixed-point KKT system (FPKKT). It consists of the perturbed KKT conditions~\eqref{equ_kkt} of a parametric linear programming problem and the fixed-point condition $\hat{\pi}=\pi$, where $\tau$ is the barrier parameter.
\begin{subequations}
  \label{equ_fpkkt}
  \begin{equation}
    \label{equ_kkt}
    \begin{bmatrix}
      \hat{\pi}_i\circ r_i-\tau_i           \\
      r_i-F_i(\pi)-v_i\mathbf{1}            \\
      \mathbf{1}^\top\hat{\pi}_i-\mathbf{1} \\
    \end{bmatrix}=0
  \end{equation}
  \begin{equation}
    \hat{\pi}=\pi
  \end{equation}
\end{subequations}

\noindent\textbf{Brouwer fixed-point form.}
Equation~\eqref{equ_brouwer} is the Brouwer function. It is given by the correspondence from $(\pi,\tau)$ to $(\hat{\pi},r,v)$ subject to the perturbed KKT conditions~\eqref{equ_kkt}. This correspondence can be shown to be a continuous function, to which the Brouwer fixed-point theorem applies.
\begin{equation}
  \label{equ_brouwer}
  (\hat{\pi},r,v)=M(\pi,\tau),\quad{\rm s.t.~perturbed~KKT~conditions~\eqref{equ_kkt}}.
\end{equation}
The Brouwer function is computed by solving the bisection problem~\eqref{equ_bisection} for every $i$.
\begin{equation}
  \label{equ_bisection}
  q_i(v_i)=\mathbf{1}^\top\frac{\tau_i}{F_i(\pi)+v_i\mathbf{1}}-1,\quad v_i\in[-\min_a F_{i,a}(\pi),-\min_a F_{i,a}(\pi)+\mathbf{1}^\top\tau_i]
\end{equation}

\noindent\textbf{Mixed complementarity problem form.}
Equation~\eqref{equ_mcp} is the (barriered) mixed complementarity problem (MCP), which is obtained from FPKKT~\eqref{equ_fpkkt} by taking the complementarity term $\pi_i\circ r_i$ as the objective function. The $\tau=0$ case is the original problem, and the $\tau>0$ case is its barrier problem.
\begin{equation}
  \label{equ_mcp}
  \begin{aligned}
    \min_{(\pi,r,v)} \quad & \sum_i\pi_i^\top r_i-\tau_i^\top\ln\pi_i-\tau_i^\top\ln r_i \\
    \textrm{s.t.}\quad     & r_i=F_i(\pi)+v_i\mathbf{1}                                  \\
                           & \mathbf{1}^\top\pi_i=1                                      \\
                           & (\pi_i,r_i)\geq0                                            \\
  \end{aligned}
\end{equation}

\noindent\textbf{Equivalence.}
The three characterizations are equivalent in the sense that for each variable-barrier pair $(\pi,\tau)$, the properties of satisfying the three characterizations are equivalent.
\begin{theorem}
  For $(\pi,\tau)\in\Delta_n^m\times\{\tau\in\mathbb{R}^{m\times n}|\tau>0\}$, the following statements are equivalent.
  \begin{enumerate}
    \item $(\pi,\tau)$ satisfies FPKKT~\eqref{equ_fpkkt} for some $r,v$.
    \item $\pi$ is a fixed point such that $\hat{\pi}=\pi$ of the Brouwer function $(\hat{\pi},r,v)=M(\pi,\tau)$.
    \item $(\pi,\tau)$ satisfies $\check{\pi}=\pi$ and the perturbed KKT conditions of MCP~\eqref{equ_mcp}, where $\check{\pi}$ is the Lagrangian multiplier associated with the constraint $r\geq0$.
  \end{enumerate}
\end{theorem}

Analogously to Paper 1, FPKKT~\eqref{equ_fpkkt} characterizes the paths leading to solutions of ${\rm VI}(F,\Delta_n^m)$, the Brouwer function $(\hat{\pi},r,v)=M(\pi,\tau)$ guarantees the existence of paths subject to FPKKT~\eqref{equ_fpkkt} via the Brouwer fixed-point theorem, and MCP~\eqref{equ_mcp} characterizes the paths as its special central paths. It can be shown that these statements indeed hold for ${\rm VI}(F,\Delta_n^m)$, but we only use the formality and the equivalence of the three characterizations. In this paper, we intend to reuse the theoretical guarantees in Paper 1 by establishing an equivalence between ${\rm VI}(F,\Delta_n^m)$ and some ${\rm VI}(\bar{F},\Delta)$ to which Paper 1 applies, instead of transferring the theoretical guarantees directly onto ${\rm VI}(F,\Delta_n^m)$.

\subsection{The composite fixed-point bundle}

As in Paper 1, the solution space of FPKKT~\eqref{equ_fpkkt} can be constructed as a fiber bundle called the fixed-point bundle, which is the central geometric object of the path-following framework.

For every $\pi\in\Delta_n^m$, the admissible $\tau$ satisfying FPKKT~\eqref{equ_fpkkt} can be expressed as
$$\tau_i=\pi_i\circ(F_i(\pi)+v_i\mathbf{1})=\pi_i\circ F_i(\pi)-(\pi_i^\top F_i(\pi))\pi_i+(v_i+\pi_i^\top F_i(\pi))\pi_i.$$
Then, setting $\tilde{\tau}_i(\pi):=\pi_i\circ F_i(\pi)-(\pi_i^\top F_i(\pi))\pi_i$ and $\bar{v}w_i:=v_i+\pi_i^\top F_i(\pi)$, part of the solution space of FPKKT~\eqref{equ_fpkkt} has the structure shown in equation~\eqref{equ_solution_space}.
\begin{equation}
  \label{equ_solution_space}
  \begin{aligned}
     & E=\bigcup_{\pi\in\Delta_n^m}\{\pi\}\times B(\pi)                                                      \\
     & B(\pi)=\bigcup_{w\in\Delta_m^\circ}\tilde{B}(\pi,w)                                                   \\
     & \tilde{B}(\pi,w)=\left\{\tilde{\tau}(\pi)+\bar{v}w\odot\pi|\bar{v}\in\mathbb{R}\setminus\{0\}\right\} \\
  \end{aligned}
\end{equation}

Note that $E$ is not the entire solution space of FPKKT~\eqref{equ_fpkkt}; it only contains those $\tau$ such that $\mathbf{1}^\top\tau_i=\bar{v}w_i$ with $\bar{v}\neq0$ and $w$ ranging over $\Delta_m^\circ$, which indicates that the $\mathbf{1}^\top\tau_i$ have the same sign for all $i$. We denote this set as
$$\mathbb{R}^{m\times n}_{\pm}:=\{\tau\in\mathbb{R}^{m\times n}|(\mathbf{1}^\top\tau_i)(\mathbf{1}^\top\tau_j)>0,~\forall i,j\},$$
then for every $(\pi,\tau)\in E$, we have $\tau\in\mathbb{R}^{m\times n}_{\pm}$.

As equation~\eqref{equ_solution_space} shows, $E$ is a fiber bundle with projection $\alpha:E\to\Delta_n^m$ onto its base space $\Delta_n^m$ such that $\alpha((\pi,\tau))=\pi$, where the fiber over $\pi\in\Delta_n^m$ is $\alpha^{-1}(\pi)=\{\pi\}\times B(\pi)$. This resembles the fiber bundle structure in Paper 1. In our case, equation~\eqref{equ_solution_space} additionally reveals a composite fiber bundle structure, in which each fiber $B(\pi)$ is itself a fiber bundle.

\begin{lemma}
  \label{thm_disjoint}
  For any $\pi\in\Delta_n^m$ and $w_1,w_2\in\Delta_m^\circ$, if $w_1\neq w_2$ then $\tilde{B}(\pi,w_1)\cap \tilde{B}(\pi,w_2)=\varnothing$.
\end{lemma}
\begin{proof}
  Suppose that for $(\pi,w_1)$ and $(\pi,w_2)$ there exist $\bar{v}_1\neq0$ and $\bar{v}_2\neq0$ such that $\pi_i\circ F_i(\pi)+\bar{v}_1w_{1,i}\pi_i=\pi_i\circ F_i(\pi)+\bar{v}_2w_{2,i}\pi_i$ for every $i$. Then $\bar{v}_1w_{1,i}\pi_i=\bar{v}_2w_{2,i}\pi_i$. Multiplying by $\mathbf{1}^\top$ and summing over $i$, we have $\bar{v}_1w_{1,i}=\bar{v}_2w_{2,i}$ and $\bar{v}_1=\bar{v}_2$. Since $\bar{v}_1\neq0$ and $\bar{v}_2\neq0$, we have $w_1=w_2$. By contraposition, if $w_1\neq w_2$, then $\tilde{B}(\pi,w_1)\cap \tilde{B}(\pi,w_2)=\varnothing$.
\end{proof}

Lemma~\ref{thm_disjoint} shows that for every $\pi$, $B(\pi)$ is indeed a fiber bundle. Since $\bar{v}=0$ is excluded from $\tilde{B}(\pi,w)$, $B(\pi)$ is the disjoint union of $\tilde{B}(\pi,w)$ over $w\in\Delta_m^\circ$, where $\tilde{B}(\pi,w)$ are the fibers of $B(\pi)$. If $\bar{v}=0$ were not excluded, all the $\tilde{B}(\pi,w)$ over $w\in\Delta_m^\circ$ would share a common point $\tilde{\tau}(\pi)$; then $\tilde{B}(\pi,w)$ would no longer be fibers of $B(\pi)$, although $B(\pi)$ could still be fibers of $E$.

Therefore, $E$ in equation~\eqref{equ_solution_space} is a composite fiber bundle, as equation~\eqref{equ_comp_fpbund} shows. We refer to this as the composite fixed-point bundle.
\begin{equation}
  \label{equ_comp_fpbund}
  \begin{aligned}
     & E\xrightarrow{\alpha_1}\Delta_n^m\times\Delta_m^\circ\xrightarrow{\alpha_2}\Delta_n^m \\
     & \alpha_1\left((\pi,\tau)\right)=(\pi,{\rm norm}([\mathbf{1}^\top\tau_i]))             \\
     & \alpha_2\left((\pi,w)\right)=\pi                                                      \\
  \end{aligned}
\end{equation}

There are three fiber bundles in the composite fixed-point bundle $E\to\Delta_n^m\times\Delta_m^\circ\to\Delta_n^m$.
\begin{itemize}
  \item Bundle $E\to\Delta_n^m\times\Delta_m^\circ$: Each fiber $\{\pi\}\times\tilde{B}(\pi,w)$ is a one-dimensional affine line spanned by $\bar{v}\in\mathbb{R}$, excluding a single point $(\pi,\tilde{\tau}(\pi))$. The projection map $\alpha_1$ projects each $(\pi,\tau)\in E$ to $(\pi,{\rm norm}([\mathbf{1}^\top\tau_i]))\in\Delta_n^m\times\Delta_m^\circ$.
  \item Bundle $\Delta_n^m\times\Delta_m^\circ\to\Delta_n^m$: Each fiber $\Delta_m^\circ$ is the interior of a simplex. The projection map $\alpha_2$ projects each $(\pi,w)\in\Delta_n^m\times\Delta_m^\circ$ to $\pi\in\Delta_n^m$.
  \item Bundle $E\to\Delta_n^m$: Each fiber $\{\pi\}\times B(\pi)$ is an $m$-dimensional affine cone spanned by $\bar{v}w$ with $w\in\Delta_m^\circ$ and $\bar{v}\in\mathbb{R}$, excluding a single point $(\pi,\tilde{\tau}(\pi))$. The projection map $\alpha_2\circ\alpha_1$ projects each $(\pi,\tau)\in E$ to $\pi\in\Delta_n^m$.
\end{itemize}

In short, $E\to\Delta_n^m\times\Delta_m^\circ$ has a large base space and small fibers, whereas $E\to\Delta_n^m$ has a small base space and large fibers. Each fiber $\{\pi\}\times B(\pi)$ of $E\to\Delta_n^m$ is the union of all the fibers $\{\pi\}\times\tilde{B}(\pi,w)$ of $E\to\Delta_n^m\times\Delta_m^\circ$ over $w\in\Delta_m^\circ$.

As in Paper 1, we then derive the fixed-point bundle equation and define two sections of the fiber bundle. The fixed-point bundle equation $G(\pi,\tau)=0$ is obtained by eliminating $r$ and $v$ in FPKKT~\eqref{equ_fpkkt} with the projection matrix $I-\pi_i\mathbf{1}^\top$, where $G(\pi,\tau)$ is given by equation~\eqref{equ_g}.
\begin{equation}
  \label{equ_g}
  G_i(\pi,\tau)=\left(I-\pi_i\mathbf{1}^\top\right)\left(\pi_i\circ F_i(\pi)-\tau_i\right),\quad(\pi,\tau)\in\Delta_n^m\times\mathbb{R}^{m\times n}_{\pm}
\end{equation}

The two sections $\check{\tau}(\pi)$ and $\tilde{\tau}(\pi)$ are sections of the bundle $E\cup\{(\pi,\tilde{\tau}(\pi))|\pi\in\Delta_n^m\}\to\Delta_n^m$. The bundle $E\cup\{(\pi,\tilde{\tau}(\pi))|\pi\in\Delta_n^m\}\to\Delta_n^m$ adds back the point excluded by $\bar{v}\neq0$ in each fiber of $E\to\Delta_n^m$. The section $\check{\tau}(\pi)$ satisfies $\check{\tau}(\pi)\geq0$ and characterizes the gap function of ${\rm VI}(F,\Delta_n^m)$ through ${\rm gap}_i(\pi_i,\pi_{\hat{i}})=\mathbf{1}^\top\check{\tau}_i(\pi)$. The section $\tilde{\tau}(\pi)$ satisfies $\mathbf{1}^\top\tilde{\tau}_i(\pi)=0$ and governs the differential properties of the composite fixed-point bundle, as discussed later.
\begin{equation}
  \label{equ_section}
  \begin{aligned}
    \check{\tau}_i(\pi) & =\pi_i\circ\left(F_i(\pi)-\left(\min_a F_{i,a}(\pi)\right)\mathbf{1}\right)                                                                   \\
    \tilde{\tau}_i(\pi) & =\pi_i\circ\left(F_i(\pi)-\left(\pi_i^\top F_i(\pi)\right)\mathbf{1}\right)=\left(I-\pi_i\mathbf{1}^\top\right)\left(\pi_i\circ F_i(\pi)\right) \\
  \end{aligned}
\end{equation}

\section{Equivalence between VI on simplex product and VI on simplex}

The equivalence between ${\rm VI}(F,\Delta_n^m)$ and ${\rm VI}(F_i(\cdot,\pi_{\hat{i}}),\Delta_n)$ in the previous section relies on a fixed parameter $\pi_{\hat{i}}$, so it holds only pointwise for each individual $\pi$ and does not apply when $\pi$ varies or in global situations. In this section, we construct a ${\rm VI}(\bar{F},\Delta_{mn}^+)$ that preserves the solutions of ${\rm VI}(F,\Delta_n^m)$, and we establish the equivalence between their fixed-point bundle frameworks. The fixed-point bundle framework can be defined for ${\rm VI}(\bar{F},\Delta_{mn}^+)$, but any theoretical guarantee that Paper 1 provides requires an additional extension argument in the next section to transfer.

\subsection{Transforming \texorpdfstring{\({\rm VI}(F,\Delta_n^m)\)}{VI on simplex product} into \texorpdfstring{\({\rm VI}(\bar{F},\Delta_{mn}^+)\)}{VI on simplex}}

We begin by constructing a VI on a simplex that preserves the solutions of ${\rm VI}(F,\Delta_n^m)$. First, we lift ${\rm VI}(F,\Delta_n^m)$ to a VI on $\Delta_n^m\times\Delta_m^\circ$ with operator $\check{F}(\pi,w)=(F(\pi),0)$. The gap function of ${\rm VI}(\check{F},\Delta_n^m\times\Delta_m^\circ)$ is derived as follows.
\begin{align*}
  {\rm gap}_{\check{F}}(\pi,w)= & \sup_{(\pi',w')\in\Delta_n^m\times\Delta_m^\circ}\left\langle(F(\pi),0),(\pi',w')-(\pi,w)\right\rangle \\
  =                             & \sup_{\pi'\in\Delta_n^m}\left\langle F(\pi),\pi'-\pi\right\rangle={\rm gap}(\pi)
\end{align*}
Thus, for any $w^\star\in\Delta_m^\circ$, $(\pi^\star,w^\star)$ is a solution of ${\rm VI}(\check{F},\Delta_n^m\times\Delta_m^\circ)$ if and only if $\pi^\star$ is a solution of ${\rm VI}(F,\Delta_n^m)$.

Viewing each $\pi_i$ as a conditional distribution and $w$ as a marginal distribution, we can transform $(\pi,w)$ into a joint distribution $\sigma$ on a simplex. This yields a real-analytic homeomorphism $\phi:\Delta_n^m\times\Delta_m^\circ\to\Delta_{mn}^+$ defined as follows.
\begin{equation}
  \label{equ_phi_base}
  \begin{aligned}
    \phi(\pi,w)=\sigma,\quad        & {\rm s.t.}~\sigma_{i,a}=w_i\pi_{i,a}                                                                           \\
    \phi^{-1}(\sigma)=(\pi,w),\quad & {\rm s.t.}~(\pi_{i,a},w_i)=\left(\sigma_{i,a}/\textstyle\sum_a\sigma_{i,a},\textstyle\sum_a\sigma_{i,a}\right) \\
  \end{aligned}
\end{equation}
We denote
$$\Delta_{mn}^+=\{\sigma\in\Delta_{mn}|\textstyle\sum_a\sigma_{i,a}>0\}.$$
The map $\phi$ is a bijection between $\Delta_n^m\times\Delta_m^\circ$ and $\Delta_{mn}^+$, because $\phi^{-1}$ as defined is its inverse. Both $\phi$ and $\phi^{-1}$ are real-analytic, since they consist only of elementary operations and the denominator $\sum_a\sigma_{i,a}>0$ on $\Delta_{mn}^+$. Thus $\phi:\Delta_n^m\times\Delta_m^\circ\to\Delta_{mn}^+$ is a real-analytic homeomorphism.

Next, we use the homeomorphism $\phi:\Delta_n^m\times\Delta_m^\circ\to\Delta_{mn}^+$ to transform ${\rm VI}(\check{F},\Delta_n^m\times\Delta_m^\circ)$ into a VI on $\Delta_{mn}^+$, following the standard procedure for changing the variable of a VI \cite{normal_cone}. For a general variational inequality ${\rm VI}(H,K)$, a standard theorem states that $x^\star\in K$ solves ${\rm VI}(H,K)$ if and only if $-H(x^\star)\in N_K(x^\star)$, where $N_K(x)=\{v\in\mathbb{R}^n|\langle v,x'-x\rangle\leq0,\forall x'\in K\}$ is the normal cone. Another standard theorem on the change of variables for normal cones states that if $z=\psi(x)$ is a diffeomorphism between $K$ and $\psi(K)$, then $N_K(x)=(\partial\psi(x)/\partial x)^\top N_{\psi(K)}(\psi(x))$. Thus, if we define $\bar{H}(z)=(\partial\psi(x)/\partial x)^{-\top}H(x)~{\rm s.t.}~x=\psi(z)$, then $x^\star\in K$ is a solution of ${\rm VI}(H,K)$ if and only if $z^\star=\psi(x^\star)$ is a solution of ${\rm VI}(\bar{H},\psi(K))$, which is a transformation obtained by changing the variable.

Note that $\bar{H}(z)=(\partial\psi(x)/\partial x)^{-\top}H(x)$ is equivalent to $\langle\bar{H}(z),dz\rangle=\langle H(x),dx\rangle$ given $z=\psi(x)$ and $dz=(\partial\psi(x)/\partial x)dx$. We use this equation to find $\bar{F}(\sigma)$ that changes the variable of ${\rm VI}(\check{F},\Delta_n^m\times\Delta_m^\circ)$ via $\sigma=\phi(\pi,w)$, which is derived as follows.
\begin{align*}
  \langle\bar{F}(\sigma),d\sigma\rangle                                                               & =\langle\check{F}(\pi,w),d(\pi,w)\rangle         \\
  \sum_i\left\langle\bar{F}_i(\sigma),w_i d\pi_i+\pi_i dw_i\right\rangle                              & =\sum_i\left\langle F_i(\pi),d\pi_i\right\rangle \\
  \sum_i\left\langle w_i\bar{F}_i(\sigma),d\pi_i\right\rangle+\sum_i(\pi_i^\top\bar{F}_i(\sigma))dw_i & =\sum_i\left\langle F_i(\pi),d\pi_i\right\rangle
\end{align*}
The coefficients of $d\pi_i$ and $dw_i$ must be identically $0$ for the equality to hold. Then, for every index $i$, $ w_i\bar{F}_i(\sigma)-F_i(\pi)=k\mathbf{1}$ and $\pi_i^\top\bar{F}_i(\sigma)=0$. Therefore, we arrive at the function $\bar{F}(\sigma)$ in equation~\eqref{equ_F_simplex}, and ${\rm VI}(\bar{F},\Delta_{mn}^+)$ is obtained from ${\rm VI}(\check{F},\Delta_n^m\times\Delta_m^\circ)$ by changing the variable via $\sigma=\phi(\pi,w)$.
\begin{equation}
  \label{equ_F_simplex}
  \bar{F}_i(\sigma)=\frac{1}{w_i}\left(F_i(\pi)-(\pi_i^\top F_i(\pi))\mathbf{1}\right),\quad{\rm s.t.~}(\pi,w)=\phi^{-1}(\sigma).
\end{equation}

In fact, we can establish the equivalence between ${\rm VI}(\bar{F},\Delta_{mn}^+)$ and ${\rm VI}(F,\Delta_n^m)$ directly, without the two-step construction of lifting and changing variables.
\begin{theorem}
  \label{thmequ_gap}
  The gap functions of ${\rm VI}(\bar{F},\Delta_{mn}^+)$ and ${\rm VI}(F,\Delta_n^m)$ satisfy the following equation.
  \begin{equation}
    {\rm gap}_{\bar{F}}(\sigma)=\max_i\frac{{\rm gap}_i(\pi_i,\pi_{\hat{i}})}{w_i},\quad{\rm s.t.}~\sigma=\phi(\pi,w)
  \end{equation}
  Consequently, given $\sigma=\phi(\pi,w)$, ${\rm gap}_{\bar{F}}(\sigma)=0$ if and only if ${\rm gap}(\pi)=0$.
\end{theorem}
\begin{proof}
  Using $(\pi,w)=\phi^{-1}(\sigma)$, we have the following deduction.
  \begin{align*}
      & \left\langle\bar{F}(\sigma),\sigma^\prime-\sigma\right\rangle=\sum_i\left\langle\frac{1}{w_i}\left(F_i(\pi)-(\pi_i^\top F_i(\pi))\mathbf{1}\right),w_i^\prime\pi_i^\prime-w_i\pi_i\right\rangle                                     \\
    = & \sum_i\left(\left\langle\frac{1}{w_i}F_i(\pi),w_i^\prime\pi_i^\prime-w_i\pi_i\right\rangle-\frac{\pi_i^\top F_i(\pi)}{w_i}(w_i^\prime-w_i)\right)=\sum_i\frac{w_i^\prime}{w_i}\left\langle F_i(\pi),\pi_i^\prime-\pi_i\right\rangle
  \end{align*}
  Then, the gap function of ${\rm VI}(\bar{F},\Delta_{mn}^+)$ is derived as follows. The second line follows from $w_i^\prime/w_i>0$ for every $i$.
  \begin{align*}
      & {\rm gap}_{\bar{F}}(\sigma)=\sup_{\sigma'\in\Delta_{mn}^+}\left\langle\bar{F}(\sigma),\sigma^\prime-\sigma\right\rangle=\sup_{(\pi',w')\in\Delta_n^m\times\Delta_m^\circ}\sum_i\frac{w_i^\prime}{w_i}\left\langle F_i(\pi),\pi_i^\prime-\pi_i\right\rangle \\
    = & \sup_{w'\in\Delta_m^\circ}\sum_i\frac{w_i^\prime}{w_i}\sup_{\pi_i^\prime\in\Delta_n}\left\langle F_i(\pi),\pi_i^\prime-\pi_i\right\rangle=\max_i\frac{{\rm gap}_i(\pi_i,\pi_{\hat{i}})}{w_i}
  \end{align*}

  Since ${\rm gap}_i(\pi_i,\pi_{\hat{i}})\geq0$, $w_i>0$, and ${\rm gap}_{\bar{F}}(\sigma)\geq0$, we have that ${\rm gap}_{\bar{F}}(\sigma)=0$ if and only if ${\rm gap}_i(\pi_i,\pi_{\hat{i}})=0$ for every $i$, which is equivalently ${\rm gap}(\pi)=0$ by Theorem~\ref{thm_decompose}.
\end{proof}

Therefore, for any $w^\star\in\Delta_m^\circ$, $\pi^\star\in\Delta_n^m$ is a solution of ${\rm VI}(F,\Delta_n^m)$ if and only if $\sigma^\star=\phi(\pi^\star,w^\star)$ is a solution of ${\rm VI}(\bar{F},\Delta_{mn}^+)$. Thus, we have transformed ${\rm VI}(F,\Delta_n^m)$ to ${\rm VI}(\bar{F},\Delta_{mn}^+)$ while preserving the solutions.

The problem ${\rm VI}(\bar{F},\Delta_{mn}^+)$ is not a smooth VI on the closed simplex, because $\Delta_{mn}^+$ excludes points where $\mathbf{1}^\top\sigma_i=0$, and $\bar{F}(\sigma)$ blows up as $\sigma$ approaches these points. Thus, the results on the fixed-point bundle framework in Paper 1 do not apply to ${\rm VI}(\bar{F},\Delta_{mn}^+)$. Nevertheless, we can still substitute $\bar{F}(\sigma)$ into Paper 1 to define the fixed-point bundle framework for ${\rm VI}(\bar{F},\Delta_{mn}^+)$, without invoking any theoretical guarantee provided by the theorems in Paper 1.

In the remainder of this section, we establish the equivalence between the fixed-point bundle framework of ${\rm VI}(\bar{F},\Delta_{mn}^+)$ and the composite fixed-point bundle framework of ${\rm VI}(F,\Delta_n^m)$, while the transition from ${\rm VI}(\bar{F},\Delta_{mn}^+)$ to an actual smooth VI on a closed simplex is addressed in the next section.

\subsection{Isomorphism between fixed-point bundles}

First, we establish the equivalence between the fixed-point bundles of ${\rm VI}(F,\Delta_n^m)$ and ${\rm VI}(\bar{F},\Delta_{mn}^+)$. The fixed-point bundle of ${\rm VI}(F,\Delta_n^m)$ is the composite fixed-point bundle $E\to\Delta_n^m\times\Delta_m^\circ\to\Delta_n^m$ defined in equation~\eqref{equ_solution_space}. From Paper 1, the fibers of the fixed-point bundle of ${\rm VI}(\bar{F},\Delta_{mn}^+)$ are $\bar{B}'(\sigma)$, derived as follows.
\begin{align*}
  \bar{B}'(\sigma) & =\left\{\sigma\circ\bar{F}(\sigma)+v\sigma|v\in\mathbb{R}\right\}                                                                                             \\
                   & =\left\{\mu|\mu_i=w_i\pi_i\circ\frac{1}{w_i}\left(F_i(\pi)-(\pi_i^\top F_i(\pi))\mathbf{1}\right)+v w_i\pi_i,(\pi,w)=\phi^{-1}(\sigma),v\in\mathbb{R}\right\} \\
                   & =\left\{{\rm reshape}(\tilde{\tau}(\pi)+\bar{v}w\odot\pi)|(\pi,w)=\phi^{-1}(\sigma),\bar{v}\in\mathbb{R}\right\}
\end{align*}

We exclude the point given by $\bar{v}=0$ from each fiber $\bar{B}'(\sigma)$, in correspondence with the composite fixed-point bundle of ${\rm VI}(F,\Delta_n^m)$. Then the fixed-point bundle $\bar{E}\to\Delta_{mn}^+$ of ${\rm VI}(\bar{F},\Delta_{mn}^+)$ is given as in equation~\eqref{equ_fpbund_simplex}.
\begin{equation}
  \label{equ_fpbund_simplex}
  \begin{aligned}
     & \bar{E}\xrightarrow{\bar{\alpha}}\Delta_{mn}^+                                                                                                \\
     & \bar{\alpha}\left((\sigma,\mu)\right)=\sigma                                                                                                  \\
     & \bar{E}=\bigcup_{\sigma\in\Delta_{mn}^+}\{\sigma\}\times \bar{B}(\sigma)                                                                      \\
     & \bar{B}(\sigma)=\left\{{\rm reshape}(\tilde{\tau}(\pi)+\bar{v}w\odot\pi)|(\pi,w)=\phi^{-1}(\sigma),\bar{v}\in\mathbb{R}\setminus\{0\}\right\} \\
  \end{aligned}
\end{equation}

Note that $\bar{B}(\sigma)$ consists of $\mu={\rm reshape}(\tilde{\tau}(\pi)+\bar{v}w\odot\pi)$, while $\tilde{B}(\pi,w)$ consists of $\tau=\tilde{\tau}(\pi)+\bar{v}w\odot\pi$. This indicates that the bundles $E\to\Delta_n^m\times\Delta_m^\circ$ and $\bar{E}\to\Delta_{mn}^+$ have very similar fibers. We can show that there is a homeomorphism between $E$ and $\bar{E}$, given by $\Phi$ as equation~\eqref{equ_Phi_total} shows.
\begin{equation}
  \label{equ_Phi_total}
  \begin{aligned}
    \Phi(\pi,\tau)=(\sigma,\mu),      & \quad{\rm s.t.}~\sigma=\phi(\pi,{\rm norm}([\mathbf{1}^\top\tau_i])),\mu_i=\tau_i \\
    \Phi^{-1}(\sigma,\mu)=(\pi,\tau), & \quad{\rm s.t.}~\pi=\phi^{-1}_\pi(\sigma),\tau_i=\mu_i                            \\
  \end{aligned}
\end{equation}

It can be verified that $\Phi$ and $\Phi^{-1}$ are indeed inverse maps of each other between $E$ and $\bar{E}$, so $\Phi$ is a homeomorphism between $E$ and $\bar{E}$ if it is continuous. We aim, however, to prove further that $\Phi$ is a homeomorphism between larger sets containing $E$ and $\bar{E}$. Recall that for every $(\pi,\tau)\in E$, we have $\tau\in\mathbb{R}^{m\times n}_{\pm}$. Since $\mu$ in $\bar{E}$ and $\tau$ in $E$ are related by a reshape, we introduce another set
$$\mathbb{R}^{mn}_{\pm}:=\{\mu\in\mathbb{R}^{mn}|(\mathbf{1}^\top\mu_i)(\mathbf{1}^\top\mu_j)>0,~\forall i,j\}={\rm reshape}(\mathbb{R}^{m\times n}_{\pm}),$$
and study the homeomorphism $\Phi$ with $\tau\in\mathbb{R}^{m\times n}_{\pm}$ and $\mu\in\mathbb{R}^{mn}_{\pm}$.

\begin{theorem}
  \label{thmequ_totalspace}
  For $(\pi,\tau)\in\Delta_n^m\times\mathbb{R}^{m\times n}_{\pm}$ and $(\sigma,\mu)\in\Delta_{mn}^+\times\mathbb{R}^{mn}_{\pm}$, the maps $\Phi$ and $\Phi^{-1}$ satisfy
  \begin{equation}
    \begin{aligned}
      \left(\Phi^{-1}\circ\Phi\right)(\pi,\tau)   & =(\pi,\tau),                                                                                              \\
      \left(\Phi\circ\Phi^{-1}\right)(\sigma,\mu) & =\left({\rm reshape}\left({\rm norm}([\mathbf{1}^\top\mu_i])\odot\phi^{-1}_\pi(\sigma)\right),\mu\right). \\
    \end{aligned}
  \end{equation}
  Consequently, $\Phi$ is a real-analytic homeomorphism such that
  \begin{equation}
    \label{equ_homeo}
    \Delta_n^m\times\mathbb{R}^{m\times n}_{\pm}\cong\left(\Delta_{mn}^+\times\mathbb{R}^{mn}_{\pm}\right)\cap\left\{(\sigma,\mu)|\phi^{-1}_w(\sigma)={\rm norm}([\mathbf{1}^\top\mu_i])\right\}.
  \end{equation}
\end{theorem}
\begin{proof}
  For each $(\pi,\tau)\in\Delta_n^m\times\mathbb{R}^{m\times n}_{\pm}$, under the map $\Phi^{-1}\circ\Phi$, we have
  \begin{align*}
    (\pi,\tau)\mapsto\left(\phi(\pi,{\rm norm}([\mathbf{1}^\top\tau_i])),{\rm reshape}(\tau)\right)\mapsto & \left(\phi^{-1}_\pi\left(\phi(\pi,{\rm norm}([\mathbf{1}^\top\tau_i]))\right),{\rm reshape}({\rm reshape}(\tau))\right) \\
                                                                                                           & \quad=(\pi,\tau).
  \end{align*}

  For each $(\sigma,\mu)\in\Delta_{mn}^+\times\mathbb{R}^{mn}_{\pm}$, under the map $\Phi\circ\Phi^{-1}$, we have
  \begin{align*}
    (\sigma,\mu)\mapsto\left(\phi^{-1}_\pi(\sigma),{\rm reshape}(\mu)\right)\mapsto & \left(\phi\left(\phi^{-1}_\pi(\sigma),{\rm norm}([\mathbf{1}^\top\mu_i])\right),{\rm reshape}({\rm reshape}(\mu))\right) \\
                                                                                    & \quad=\left({\rm reshape}\left({\rm norm}([\mathbf{1}^\top\mu_i])\odot\phi^{-1}_\pi(\sigma)\right),\mu\right).
  \end{align*}

  The map $\Phi^{-1}\circ\Phi$ is always the identity map, whereas $\Phi\circ\Phi^{-1}$ is the identity map if and only if ${\rm norm}([\mathbf{1}^\top\mu_i])=\phi^{-1}_w(\sigma)$. To prove that $\Phi$ is a bijection between the left-hand side (LHS) and the right-hand side (RHS) of equation~\eqref{equ_homeo}, we further need to show that $\Phi({\rm LHS})\subset{\rm RHS}$ and $\Phi^{-1}({\rm RHS})\subset{\rm LHS}$.

  For every $(\sigma,\mu)=\Phi(\pi,\tau)$ with $(\pi,\tau)\in\Delta_n^m\times\mathbb{R}^{m\times n}_{\pm}$ (belonging to the LHS), by definition of $\Phi$, we have
  \begin{align*}
    \phi^{-1}_w(\sigma)=\phi^{-1}_w\left(\phi(\pi,{\rm norm}([\mathbf{1}^\top\tau_i]))\right)={\rm norm}([\mathbf{1}^\top\tau_i])={\rm norm}([\mathbf{1}^\top\mu_i]).
  \end{align*}
  Then, $(\sigma,\mu)$ belongs to the RHS. Thus, $\Phi({\rm LHS})\subset{\rm RHS}$. For every $(\pi,\tau)=\Phi^{-1}(\sigma,\mu)$ with $(\sigma,\mu)$ belonging to the RHS, we trivially have $(\pi,\tau)\in\Delta_n^m\times\mathbb{R}^{m\times n}_{\pm}$ (belonging to the LHS). Thus, $\Phi^{-1}({\rm RHS})\subset{\rm LHS}$. Therefore, we have proved that $\Phi$ is a bijection between the LHS and RHS in equation~\eqref{equ_homeo}.

  The maps $\phi$ and $\phi^{-1}$ in equation~\eqref{equ_phi_base} are real-analytic because they consist only of elementary operations and the denominator $\sum_a\sigma_{i,a}>0$ on $\Delta_{mn}^+$. The maps $\Phi$ and $\Phi^{-1}$ in equation~\eqref{equ_Phi_total} are also real-analytic because they consist only of elementary operations and the real-analytic $\phi$ and $\phi^{-1}$, and the denominator $\sum_i \mathbf{1}^\top\tau_i\neq0$ on $\mathbb{R}^{m\times n}_{\pm}$. Therefore, $\Phi$ is a real-analytic homeomorphism between the LHS and RHS in equation~\eqref{equ_homeo}.
\end{proof}

Theorem~\ref{thmequ_totalspace} shows that $\Phi$ is a homeomorphism beyond $E$ and $\bar{E}$. This is useful for establishing the equivalence between the three equivalent characterizations of ${\rm VI}(F,\Delta_n^m)$ and ${\rm VI}(\bar{F},\Delta_{mn}^+)$, because $E$ and $\bar{E}$ contain only the solutions of the characterizations, whereas with $\Phi$ in equation~\eqref{equ_homeo} we can compare the values of the characterizations both on and off the solutions.

Next, we show that $E\to\Delta_n^m\times\Delta_m^\circ$ and $\bar{E}\to\Delta_{mn}^+$ are isomorphic as fiber bundles. Given general fiber bundles $E\xrightarrow{\alpha}B$ and $E'\xrightarrow{\alpha'}B'$, a bundle isomorphism is a pair of homeomorphisms $(\psi:B\to B',\Psi:E\to E')$ such that $\alpha'\circ\Psi=\psi\circ\alpha$.

\begin{theorem}
  \label{thmequ_bundle}
  $(\Phi,\phi)$ is a bundle isomorphism such that
  $$E\to\Delta_n^m\times\Delta_m^\circ\cong\bar{E}\to\Delta_{mn}^+.$$
\end{theorem}
\begin{proof}
  First, we prove that $\Phi$ is a homeomorphism between $E$ and $\bar{E}$. Trivially, $E\subset\Delta_n^m\times\mathbb{R}^{m\times n}_{\pm}$. By the definition of $\bar{B}(\sigma)$, every $\mu\in\bar{B}(\sigma)$ for some $\sigma\in\Delta_{mn}^+$ satisfies $\mathbf{1}^\top\mu_i=\bar{v}w_i$ with $w=\phi^{-1}_w(\sigma)$. Thus, $\bar{E}\subset(\Delta_{mn}^+\times\mathbb{R}^{mn}_{\pm})\cap\{(\sigma,\mu)|\phi^{-1}_w(\sigma)={\rm norm}([\mathbf{1}^\top\mu_i])\}$. It remains to show that $\Phi(E)\subset\bar{E}$ and $\Phi^{-1}(\bar{E})\subset E$.

  For every $(\pi,\tau)\in E$, we have $\tau=\tilde{\tau}(\pi)+\bar{v}w\odot\pi$ for some $w\in\Delta_m^\circ$ and $\bar{v}\neq0$. Then, under the map $\Phi$, we have
  \begin{align*}
    (\pi,\tilde{\tau}(\pi)+\bar{v}w\odot\pi)\mapsto & \left(\phi(\pi,w),{\rm reshape}(\tilde{\tau}(\pi)+\bar{v}w\odot\pi)\right)                                            \\
                                                    & \quad=\left(\sigma,{\rm reshape}(\tilde{\tau}(\pi)+\bar{v}w\odot\pi)\right)\quad{\rm s.t.}~(\pi,w)=\phi^{-1}(\sigma),
  \end{align*}
  which indicates that $\Phi(\pi,\tau)\in\bar{E}$. For every $(\sigma,\mu)\in \bar{E}$, we have $\mu={\rm reshape}(\tilde{\tau}(\pi)+\bar{v}w\odot\pi)$ for $(\pi,w)=\phi^{-1}(\sigma)$ and $\bar{v}\neq0$. Then, under the map $\Phi^{-1}$, we have
  \begin{align*}
    (\sigma,{\rm reshape}(\tilde{\tau}(\pi)+\bar{v}w\odot\pi))\mapsto & \left(\phi^{-1}_\pi(\sigma),\tilde{\tau}(\pi)+\bar{v}w\odot\pi\right)                               \\
                                                                      & \quad=\left(\pi,\tilde{\tau}(\pi)+\bar{v}w\odot\pi\right)\quad{\rm s.t.}~(\pi,w)=\phi^{-1}(\sigma),
  \end{align*}
  which indicates that $\Phi^{-1}(\sigma,\mu)\in E$. Therefore, we have $\Phi(E)\subset\bar{E}$ and $\Phi^{-1}(\bar{E})\subset E$, so $\Phi$ is a homeomorphism between $E$ and $\bar{E}$.

  We already have that $\phi$ is a homeomorphism between $\Delta_n^m\times\Delta_m^\circ$ and $\Delta_{mn}^+$. For $(\Phi,\phi)$ to constitute a bundle isomorphism, we additionally need $\bar{\alpha}\circ\Phi=\phi\circ\alpha_1$, where $\alpha_1$ and $\bar{\alpha}$ are bundle projections.

  For each $(\pi,\tau)\in E$, under the maps $\Phi:E\to\bar{E}$ and $\bar{\alpha}:\bar{E}\to\Delta_{mn}^+$, we have
  $$(\pi,\tau)\mapsto(\phi(\pi,{\rm norm}([\mathbf{1}^\top\tau_i])),{\rm reshape}(\tau))\mapsto\phi(\pi,{\rm norm}([\mathbf{1}^\top\tau_i])),$$
  and under the maps $\alpha_1:E\to\Delta_n^m\times\Delta_m^\circ$ and $\phi:\Delta_n^m\times\Delta_m^\circ\to\Delta_{mn}^+$, we have
  $$(\pi,\tau)\mapsto(\pi,{\rm norm}([\mathbf{1}^\top\tau_i]))\mapsto\phi(\pi,{\rm norm}([\mathbf{1}^\top\tau_i])).$$

  Therefore, we obtain $\bar{\alpha}\circ\Phi=\phi\circ\alpha_1$, so $(\Phi,\phi)$ is a bundle isomorphism between $E\to\Delta_n^m\times\Delta_m^\circ$ and $\bar{E}\to\Delta_{mn}^+$.
\end{proof}

\subsection{Equivalence between characterizations}

Having obtained the homeomorphism in equation~\eqref{equ_homeo}, we use it to establish the equivalence between the characterizations of ${\rm VI}(F,\Delta_n^m)$ and ${\rm VI}(\bar{F},\Delta_{mn}^+)$. We first introduce the objects whose equivalence is to be studied, replacing the FPKKTs by the fixed-point bundle equations and adding the sections of the fixed-point bundles.

\noindent\textbf{Sections.}
The sections $\check{\tau}(\pi)$ and $\tilde{\tau}(\pi)$ of the bundle $E\cup\{(\pi,\tilde{\tau}(\pi))|\pi\in\Delta_n^m\}\to\Delta_n^m$ are given as in equation~\eqref{equ_section}. The sections $\check{\mu}(\sigma)$ and $\tilde{\mu}(\sigma)$ of the bundle $\bar{E}\cup\{(\sigma,{\rm reshape}(\tilde{\tau}(\phi^{-1}_\pi(\sigma))))|\sigma\in\Delta_{mn}^+\}\to\Delta_{mn}^+$ are taken from Paper 1 and are shown in equation~\eqref{equ_section_simplex}.
\begin{equation}
  \label{equ_section_simplex}
  \begin{aligned}
    \check{\mu}(\sigma)= & \sigma\circ\left(\bar{F}(\sigma)-\left(\min_{i,a} \bar{F}_{i,a}(\sigma)\right)\mathbf{1}\right)                                                                         \\
    \tilde{\mu}(\sigma)= & \sigma\circ\left(\bar{F}(\sigma)-\left(\sigma^\top \bar{F}(\sigma)\right)\mathbf{1}\right)=\left(I-\sigma\mathbf{1}^\top\right)\left(\sigma\circ \bar{F}(\sigma)\right) \\
  \end{aligned}
\end{equation}

\noindent\textbf{Fixed-point bundle equations.}
The fixed-point bundle equation of $E\to\Delta_n^m$ is given as in equation~\eqref{equ_g}. The fixed-point bundle equation of $\bar{E}\to\Delta_{mn}^+$ is $\bar{G}(\sigma,\mu)=0$ from Paper 1, where
\begin{equation}
  \bar{G}(\sigma,\mu)=\left(I-\sigma\mathbf{1}^\top\right)(\sigma\circ \bar{F}(\sigma)-\mu).
\end{equation}

\noindent\textbf{Brouwer functions.}
The Brouwer function $(\hat{\pi},r,v)=M(\pi,\tau)$ characterizing ${\rm VI}(F,\Delta_n^m)$ is given as in equation~\eqref{equ_brouwer}, and the Brouwer function characterizing ${\rm VI}(\bar{F},\Delta_{mn}^+)$ is taken from Paper 1. In fact, $(\hat{\pi},r,v)=M(\pi,\tau)$ in this paper can only be shown equivalent to an index-wise variant $(\hat{\sigma},\bar{r},\bar{v})=\bar{M}(\sigma,\mu)$, as shown in equation~\eqref{equ_kkt_simplex}, rather than to the original Brouwer function in Paper 1.
\begin{equation}
  \label{equ_kkt_simplex}
  (\hat{\sigma},\bar{r},\bar{v})=\bar{M}(\sigma,\mu),\quad{\rm s.t.~}
  \begin{bmatrix}
    \hat{\sigma}_i\circ \bar{r}_i-\mu_i                   \\
    r_i-\bar{F}_i(\sigma)-\bar{v}_i\mathbf{1}             \\
    \mathbf{1}^\top\hat{\sigma}_i-\mathbf{1}^\top\sigma_i \\
  \end{bmatrix}=0
\end{equation}

Compared with the original, the index-wise variant $(\hat{\sigma},\bar{r},\bar{v})=\bar{M}(\sigma,\mu)$ has two modifications. First, in Paper 1 all the components of $\bar{r}$ depend on a single scalar $\bar{v}$, whereas in equation~\eqref{equ_kkt_simplex} each $\bar{r}_i$ depends on an independent $\bar{v}_i$. Second, the constraint $\sum_i\mathbf{1}^\top\hat{\sigma}_i=1$ in Paper 1 becomes the index-wise constraint $\mathbf{1}^\top\hat{\sigma}_i-\mathbf{1}^\top\sigma_i$ in equation~\eqref{equ_kkt_simplex}. This variant essentially uses the index-wise $\bar{v}_i$ to control each $\mathbf{1}^\top\hat{\sigma}_i$ instead of the entire sum $\sum_i\mathbf{1}^\top\hat{\sigma}_i$.

\noindent\textbf{MCPs.}
The two MCPs characterizing ${\rm VI}(F,\Delta_n^m)$ and ${\rm VI}(\bar{F},\Delta_{mn}^+)$ are given as in equation~\eqref{equ_mcp_obj}. The term ${\rm mcp}(\pi,\tau)$ originates from MCP~\eqref{equ_mcp}, and $(r,v)$ is additionally constrained by the Brouwer function $(\hat{\pi},r,v)=M(\pi,\tau)$ to make $\pi$ the only optimization variable. The term ${\rm mcp}(\sigma,\mu)$ originates from the MCP in Paper 1, where $(\bar{r},\bar{v})$ is originally additionally constrained by the Brouwer function in Paper 1 to make $\sigma$ the only optimization variable; we instead modify $(\bar{r},\bar{v})$ to be constrained by the index-wise variant $(\hat{\sigma},\bar{r},\bar{v})=\bar{M}(\sigma,\mu)$.
\begin{equation}
  \label{equ_mcp_obj}
  \begin{aligned}
    {\rm mcp}(\pi,\tau)   & =\sum_i\left(\pi_i^\top r_i-\tau_i^\top\ln\pi_i-\tau_i^\top \ln r_i\right),\quad{\rm s.t.~}(\hat{\pi},r,v)=M(\pi,\tau)         \\
    {\rm mcp}(\sigma,\mu) & =\sigma^\top \bar{r}-\mu^\top\ln\sigma-\mu^\top \ln \bar{r},\quad{\rm s.t.~}(\hat{\sigma},\bar{r},\bar{v})=\bar{M}(\sigma,\mu) \\
  \end{aligned}
\end{equation}

\noindent\textbf{Equivalence.}
The sections $\check{\tau}(\pi)$ and $\check{\mu}(\sigma)$ do not themselves have a useful relation, but, since ${\rm gap}_i(\pi_i,\pi_{\hat{i}})=\mathbf{1}^\top\check{\tau}_i(\pi)$ and ${\rm gap}_{\bar{F}}(\sigma)=\mathbf{1}^\top\check{\mu}(\sigma)$, they are related through the gap function relation in Theorem~\ref{thmequ_gap}. The relations between the remaining characterizations are given by the following proposition.

\begin{proposition}
  \label{thmequ_chara}
  For $(\pi,\tau)\in\Delta_n^m\times\mathbb{R}^{m\times n}_{\pm}$ and $(\sigma,\mu)\in\Delta_{mn}^+\times\mathbb{R}^{mn}_{\pm}$, let $\phi^{-1}_w(\sigma)={\rm norm}([\mathbf{1}^\top\mu_i])$ and $(\sigma,\mu)=\Phi(\pi,\tau)$, then the following equalities hold.
  \begin{enumerate}
    \item $\tilde{\mu}_i(\sigma)=\tilde{\tau}_i(\pi)$.
    \item $\bar{G}_i(\sigma,\mu)=G_i(\pi,\tau)$.
    \item $\hat{\sigma}_i=w_i\hat{\pi}_i$ and $\bar{r}_i=r_i/w_i$, where $(\hat{\pi},r,v)=M(\pi,\tau)$ and $(\hat{\sigma},\bar{r},\bar{v})=\bar{M}(\sigma,\mu)$.
    \item ${\rm mcp}(\sigma,\mu)={\rm mcp}(\pi,\tau)$.
  \end{enumerate}
\end{proposition}
\begin{proof}
  (i) Substituting $\sigma=\phi(\pi,w)$ and $\bar{F}(\sigma)$, we have
  $$\left(\sigma\circ\bar{F}(\sigma)\right)_i=w_i\pi_i\circ\frac{1}{w_i}\left(F_i(\pi)-(\pi_i^\top F_i(\pi))\mathbf{1}\right)=\tilde{\tau}_i(\pi).$$
  Then, since $\mathbf{1}^\top\tilde{\tau}_i(\pi)=0$ for every index $i$, we obtain the equation
  \begin{align*}
    \tilde{\mu}_i(\sigma)=\left(\sigma\circ\left(\bar{F}(\sigma)-\left(\sigma^\top \bar{F}(\sigma)\right)\mathbf{1}\right)\right) _i=\tilde{\tau}_i(\pi)-\left(\textstyle\sum_j\mathbf{1}^\top\tilde{\tau}_j(\pi)\right)\sigma_i=\tilde{\tau}_i(\pi).
  \end{align*}

  (ii) Substituting $(\sigma,\mu)=\Phi(\pi,\tau)$ and $\sigma=\phi(\pi,w)$, using $\tilde{\mu}_i(\sigma)=\tilde{\tau}_i(\pi)$ in the first line, and using $w=\phi^{-1}_w(\sigma)={\rm norm}([\mathbf{1}^\top\mu_i])={\rm norm}([\mathbf{1}^\top\tau_i])$ in the second line, we have
  \begin{align*}
    \bar{G}_i(\sigma,\mu) & =\tilde{\mu}_i(\sigma)-\left(\mu_i-(\mathbf{1}^\top\mu)\sigma_i\right)=\tilde{\tau}_i(\pi)-\left(\tau_i-(\textstyle\sum_i\mathbf{1}^\top\tau_i)w_i\pi_i\right) \\
                          & =\tilde{\tau}_i(\pi)-\left(\tau_i-(\mathbf{1}^\top\tau_i)\pi_i\right)=G_i(\pi,\tau).
  \end{align*}

  (iii) In $(\hat{\pi},r,v)=M(\pi,\tau)$, each $v_i$ is given by the zero of $q_i(v_i)=0$ as equation~\eqref{equ_bisection} shows. In $(\hat{\sigma},\bar{r},\bar{v})=\bar{M}(\sigma,\mu)$, each $\bar{v}_i$ is given by the zero of $\bar{q}_i(\bar{v}_i)=0$ via a similar derivation from equation~\eqref{equ_kkt_simplex}.
  \begin{align*}
    q_i(v_i)             & =\mathbf{1}^\top\frac{\tau_i}{F_i(\pi)+v_i\mathbf{1}}-1                                                                                                                                                    \\
    \bar{q}_i(\bar{v}_i) & =\mathbf{1}^\top\frac{\mu_i}{\bar{F}_i(\sigma)+\bar{v}_i\mathbf{1}}-\mathbf{1}^\top\sigma_i=w_i\left(\mathbf{1}^\top\frac{\mu_i}{F_i(\pi)-(\pi_i^\top F_i(\pi))\mathbf{1}+\bar{v}_iw_i\mathbf{1}}-1\right)
  \end{align*}
  Comparing $q_i(v_i)$ and $\bar{q}_i(\bar{v}_i)$, it can be directly observed that when $\sigma=\phi(\pi,w)$ and $\mu_i=\tau_i$, their zeros $v_i$ and $\bar{v}_i$ satisfy $v_i=\bar{v}_iw_i-\pi_i^\top F_i(\pi)$. Then, we obtain $\bar{r}_i=r_i/w_i$ at this time from $r_i=F_i(\pi)+v_i\mathbf{1}$ and $\bar{r}_i=\bar{F}_i(\sigma)+\bar{v}_i\mathbf{1}$. Then, $\hat{\sigma}_i=w_i\hat{\pi}_i$ follows from $\hat{\pi}_i=\tau_i/r_i$ and $\hat{\sigma}_i=\mu_i/\bar{r}_i$.

  (iv) By $\bar{r}_i=r_i/w_i$ from the Brouwer functions and $\sigma_i=w_i\pi_i$, we have $\sigma_i\circ\bar{r}_i=\pi_i\circ r_i$, then it follows that ${\rm mcp}(\sigma,\mu)={\rm mcp}(\pi,\tau)$.
\end{proof}

Proposition~\ref{thmequ_chara} will be used to transform the path-following iterations of the fixed-point bundle framework from ${\rm VI}(\bar{F},\Delta_{mn}^+)$ to ${\rm VI}(F,\Delta_n^m)$. We will also handle the difference between the index-wise variants and the original Brouwer functions and MCPs.

\subsection{Equivalence between differential properties}

In Paper 1, to maintain the constraint $\sigma\in\Delta$, the variable $\sigma$ is represented by the softmax function as $\sigma={\rm softmax}(\theta)$, and differentiable properties are then measured by differentiating with respect to $\theta$ instead of $\sigma$. In this subsection we derive the equivalence between the differential properties of the two fixed-point bundles. First, we likewise use the softmax function to represent $\pi$, so as to maintain the constraint $\pi\in\Delta_n^m$.

\begin{theorem}
  \label{thm_softmax}
  For any $\sigma\in\Delta_{mn}^+$, $\pi\in\Delta_n^m$, and $w\in\Delta_m^\circ$, $\sigma=\phi(\pi,w)$ if and only if there exists $\theta$ such that
  \begin{equation}
    \label{equ_softmax}
    \sigma={\rm softmax}(\theta)=\frac{\exp(\theta)}{\sum_i\mathbf{1}^\top\exp(\theta_i)},~w_i=\frac{\mathbf{1}^\top\exp(\theta_i)}{\sum_i\mathbf{1}^\top\exp(\theta_i)},~\pi_i={\rm softmax}(\theta_i)=\frac{\exp(\theta_i)}{\mathbf{1}^\top\exp(\theta_i)}.
  \end{equation}
\end{theorem}
\begin{proof}
  Given $\sigma=\phi(\pi,w)$, we have the softmax representations with $\theta=\ln\sigma$. Given the softmax representations, $\sigma_i=w_i\pi_i$ can be directly verified.
\end{proof}

Theorem~\ref{thm_softmax} shows that if $\sigma,\pi,w$ are represented by softmax functions, then $\sigma=\phi(\pi,w)$ holds automatically. If in addition $\mu={\rm reshape}(\tau)$, then $(\sigma,\mu)=\Phi(\pi,\tau)$ holds as long as $w(\theta)={\rm norm}([\mathbf{1}^\top\tau_i])$. Consequently, viewed as functions of $\theta$, the sections $\tilde{\tau}(\pi)$ and $\tilde{\mu}(\sigma)$ in Proposition~\ref{thmequ_chara}~(i) are in fact the same functions, and, viewed as functions of $(\theta,\tau)$, the fixed-point bundle functions $\bar{G}(\sigma,\mu)$ and $G(\pi,\tau)$ in Proposition~\ref{thmequ_chara}~(ii) are in fact the same functions as long as $w(\theta)={\rm norm}([\mathbf{1}^\top\tau_i])$.

Since differentiable properties are measured by differentiating with respect to $\theta$, we derive the differentials of $\sigma,\pi,w$ with respect to $\theta$.
\begin{align*}
  d\sigma= & \frac{\exp(\theta)\circ d\theta}{\mathbf{1}^\top\exp(\theta)}-\frac{\mathbf{1}^\top\left(\exp(\theta)\circ d\theta\right)}{(\mathbf{1}^\top\exp(\theta))^2} \exp(\theta)=\sigma\circ d\theta-(\sigma^\top d\theta)\sigma                                                                                                \\
  d\pi_i=  & \frac{\exp(\theta_i)\circ d\theta_i}{\mathbf{1}^\top\exp(\theta_i)}-\frac{\mathbf{1}^\top\left(\exp(\theta_i)\circ d\theta_i\right)}{(\mathbf{1}^\top\exp(\theta_i))^2} \exp(\theta_i)=\pi_i\circ d\theta_i-(\pi_i^\top d\theta_i)\pi_i                                                                                 \\
  dw_i=    & \frac{\mathbf{1}^\top\left(\exp(\theta_i)\circ d\theta_i\right)}{\sum_j\mathbf{1}^\top\exp(\theta_j)}-\frac{\sum_j\mathbf{1}^\top\left(\exp(\theta_j)\circ d\theta_j\right)}{(\sum_j\mathbf{1}^\top\exp(\theta_j))^2}(\mathbf{1}^\top\exp(\theta_i))=w_i\pi_i^\top d\theta_i-w_i\textstyle\sum_jw_j\pi_j^\top d\theta_j
\end{align*}

The derivation shows that the differentials $d\sigma,d\pi,dw$ are as equation~\eqref{equ_differential} shows.
\begin{equation}
  \label{equ_differential}
  d\sigma/\sigma=\left(I-\mathbf{1}\sigma^\top\right)d\theta,~d\pi_i/\pi_i=\left(I-\mathbf{1}\pi_i^\top\right)d\theta_i,~dw/w=\left(I-\mathbf{1}w^\top\right)[\pi_i^\top d\theta_i]
\end{equation}

Note that ${\rm softmax}(\theta+k\mathbf{1})={\rm softmax}(\theta)$, so softmax is invariant under adding a scaled $\mathbf{1}$, and adding a scaled $\mathbf{1}$ to $d\theta$ does not change the differentials $d\sigma$ or $d\pi$. We therefore denote $\overline{d\theta}$ and $\widetilde{d\theta}$ as in equation~\eqref{equ_dtheta}, so that $\sigma^\top\overline{d\theta}=0$ and $\pi_i^\top\widetilde{d\theta}_i=0$. Here $\overline{d\theta}$ has already appeared in Paper 1.
\begin{equation}
  \label{equ_dtheta}
  \overline{d\theta}:=\left(I-\mathbf{1}\sigma^\top\right)d\theta,\quad\widetilde{d\theta}_i:=\left(I-\mathbf{1}\pi_i^\top\right)d\theta_i
\end{equation}

As in Paper 1, instead of the plain $d\theta$, we actually use $\overline{d\theta}$ and $\widetilde{d\theta}$ as the differentials measuring the differentiable properties of the fixed-point bundles. We first study their relation.

\begin{lemma}
  \label{thm_dtheta}
  Let $(\sigma,\pi,w)$ be subject to equation~\eqref{equ_softmax}, then for any $d\theta$, the following statements hold.
  \begin{enumerate}
    \item $\sigma^\top\widetilde{d\theta}=0$.
    \item $\pi_i^\top\overline{d\theta}_i=0$ for every index $i$ if and only if $\widetilde{d\theta}=\overline{d\theta}$, which is equivalently $dw=0$.
  \end{enumerate}
\end{lemma}
\begin{proof}
  (i) For $\sigma^\top\widetilde{d\theta}$, we can derive $\sigma^\top\widetilde{d\theta}=\textstyle\sum_jw_j\pi_j^\top\widetilde{d\theta}_j$. Then, we have $\sigma^\top\widetilde{d\theta}=0$ since $\pi_j^\top\widetilde{d\theta}_j=0$.

  (ii) For $\pi_i^\top\overline{d\theta}_i$, $\overline{d\theta}-\widetilde{d\theta}$, and $dw$, we can derive the following three equations.
  \begin{align*}
    \pi_i^\top\overline{d\theta}_i             & =\pi_i^\top\left(d\theta_i-(\sigma^\top d\theta)\mathbf{1}\right)=\pi_i^\top d\theta_i-\textstyle\sum_jw_j\pi_j^\top d\theta_j                                               \\
    \overline{d\theta}_i-\widetilde{d\theta}_i & =d\theta_i-(\sigma^\top d\theta)\mathbf{1}-(d\theta_i-(\pi_i^\top d\theta_i)\mathbf{1})=(\pi_i^\top d\theta_i)\mathbf{1}-(\textstyle\sum_jw_j\pi_j^\top d\theta_j)\mathbf{1} \\
    dw/w                                       & =\left(I-\mathbf{1}w^\top\right)[\pi_i^\top d\theta_i]=\left[\pi_i^\top d\theta_i-\textstyle\sum_jw_j\pi_j^\top d\theta_j\right].
  \end{align*}

  Therefore, $\pi_i^\top\overline{d\theta}_i=0$ for every index $i$, $\widetilde{d\theta}=\overline{d\theta}$, and $dw=0$ are all equivalent to $\pi_i^\top d\theta_i=\sum_jw_j\pi_j^\top d\theta_j$ for every index $i$.
\end{proof}

The differentiable properties of the fixed-point bundle are governed by the sections $\tilde{\mu}(\sigma)$ and $\tilde{\tau}_i(\pi)$, whose differentials appear in the singular manifold, the differential equation, the Newton equation, and the gradient of related functions. In Paper 1, $d\tilde{\mu}(\sigma)$ is measured with respect to $\overline{d\theta}$. We next measure $d\tilde{\tau}_i(\pi)$ with respect to $\widetilde{d\theta}$ and derive its relation to $d\tilde{\mu}(\sigma)$, since $\tilde{\mu}(\sigma)=\tilde{\tau}_i(\pi)$. In the following equations, let $\delta_{ij}$ be the Kronecker delta, with $\delta_{ij}=1$ when $i=j$ and $\delta_{ij}=0$ when $i\neq j$.

\begin{theorem}
  \label{thmequ_diff}
  Let $\sigma=\phi(\pi,w)$, then the differentials of $\tilde{\tau}(\pi)$ and $\tilde{\mu}(\sigma)$ satisfy
  \begin{equation}
    d\tilde{\tau}_i(\pi)=d\tilde{\mu}_i(\sigma)=\sigma_i\circ\bar{J}_i(\sigma)\overline{d\theta}=\sigma_i\circ\bar{J}_i(\sigma)\widetilde{d\theta}=\pi_i\circ J_i(\pi)\widetilde{d\theta},
  \end{equation}
  where $J(\pi)$ and $\bar{J}(\sigma)$ are given by the following equations.
  \begin{equation}
    \label{equ_j}
    \begin{aligned}
      J_{ij}(\pi)     & :=\left(I-\mathbf{1}\pi_i^\top\right)\left(\frac{\partial F_i(\pi)}{\partial\pi_j}\circ\pi_j+\delta_{ij}\left(\left(I-\mathbf{1}\pi_i^\top\right)F_i(\pi)\right)\circ I\right)\left(I-\mathbf{1}\pi_j^\top\right)+kw_iw_j\mathbf{1}\pi_j^\top    \\
      \bar{J}(\sigma) & :=\left(I-\mathbf{1}\sigma^\top\right)\left(\frac{\partial \bar{F}(\sigma)}{\partial\sigma}\circ\sigma+\left(\left(I-\mathbf{1}\sigma^\top\right)\bar{F}(\sigma)\right)\circ I\right)\left(I-\mathbf{1}\sigma^\top\right)+k\mathbf{1}\sigma^\top \\
    \end{aligned}
  \end{equation}
  Consequently, $w_i\bar{J}_{ij}(\sigma)=J_{ij}(\pi)$.
\end{theorem}
\begin{proof}
  (1) Proving $d\tilde{\tau}_i(\pi)=d\tilde{\mu}_i(\sigma)=\sigma_i\circ\bar{J}_i(\sigma)\overline{d\theta}$.

  In Paper 1, $d\tilde{\mu}_i(\sigma)=\sigma_i\circ\bar{J}_i(\sigma)\overline{d\theta}$, and Proposition~\ref{thmequ_chara} shows that $\tilde{\tau}_i(\pi)=\tilde{\mu}_i(\sigma)$, thus $d\tilde{\tau}_i(\pi)=d\tilde{\mu}_i(\sigma)$.

  (2) Proving $d\tilde{\tau}_i(\pi)=\sigma_i\circ\bar{J}_i(\sigma)\overline{d\theta}=\sigma_i\circ\bar{J}_i(\sigma)\widetilde{d\theta}$.

  Note that $d\theta$ and $\widetilde{d\theta}$ lead to the same $d\pi$ in equation~\eqref{equ_differential}, since $I-\mathbf{1}\pi_i^\top$ is idempotent. Substituting $\widetilde{d\theta}$ for $d\theta$ therefore does not change $d\tilde{\tau}_i(\pi)$, since $\tilde{\tau}_i(\pi)$ depends only on $\pi$. Thus we have the following derivation, where $(I-\mathbf{1}\sigma^\top)\widetilde{d\theta}=\widetilde{d\theta}$ by Lemma~\ref{thm_dtheta}~(i).
  \begin{align*}
    d\tilde{\tau}_i(\pi)=\sigma_i\circ\bar{J}_i(\sigma)\overline{d\theta}=\sigma_i\circ\bar{J}_i(\sigma)\left(I-\mathbf{1}\sigma^\top\right)d\theta=\sigma_i\circ\bar{J}_i(\sigma)\left(I-\mathbf{1}\sigma^\top\right)\widetilde{d\theta}=\sigma_i\circ\bar{J}_i(\sigma)\widetilde{d\theta}
  \end{align*}

  (3) Proving $d\tilde{\tau}_i(\pi)=\pi_i\circ J_i(\pi)\widetilde{d\theta}$.

  The differential $d\tilde{\tau}(\pi)$ is derived as follows. The third line uses $\mathbf{1}^\top d\pi_i=0$, and the last line uses $d\pi_i/\pi_i=\widetilde{d\theta}_i=(I-\mathbf{1}\pi_i^\top)\widetilde{d\theta}_i$.
  \begin{align*}
     & d\tilde{\tau}_i(\pi) =\left(I-\pi_i\mathbf{1}^\top\right)\left(\pi_i\circ\sum_j\frac{\partial F_i(\pi)}{\partial\pi_j}d\pi_j+F_i(\pi)\circ d\pi_i\right)-\mathbf{1}^\top\left(\pi_i\circ F_i(\pi)\right)d\pi_i                                                                                                \\
     & =\left(I-\pi_i\mathbf{1}^\top\right)\left(\pi_i\circ\sum_j\frac{\partial F_i(\pi)}{\partial\pi_j}d\pi_j+\left(F_i(\pi)\circ I-(\pi_i^\top F_i(\pi)) I\right)d\pi_i\right)-\left(\pi_i^\top F_i(\pi)\right)\pi_i\mathbf{1}^\top d\pi_i                                                                         \\
     & =\pi_i\circ\left(I-\mathbf{1}\pi_i^\top\right)\left(\sum_j\left(\frac{\partial F_i(\pi)}{\partial\pi_j}\circ\pi_j\right)\frac{d\pi_j}{\pi_j}+\left(\left(I-\mathbf{1}\pi_i^\top\right)F_i(\pi)\right)\circ\frac{d\pi_i}{\pi_i}\right)                                                                         \\
     & =\pi_i\circ\left(I-\mathbf{1}\pi_i^\top\right)\left(\sum_j\left(\frac{\partial F_i(\pi)}{\partial\pi_j}\circ\pi_j\right)\left(I-\mathbf{1}\pi_j^\top\right)\widetilde{d\theta}_j+\left(\left(I-\mathbf{1}\pi_i^\top\right)F_i(\pi)\right)\circ\left(I-\mathbf{1}\pi_i^\top\right)\widetilde{d\theta}_i\right)
  \end{align*}
  Substituting $J_{ij}(\pi)$ and using $\pi_i^\top\widetilde{d\theta}_i=0$, we obtain $d\tilde{\tau}_i(\pi)=\pi_i\circ J_i(\pi)\widetilde{d\theta}$.

  (4) Proving $w_i\bar{J}_{ij}(\sigma)=J_{ij}(\pi)$.

  Let $\bar{J}_{0ij}(\sigma)$ and $J_{0ij}(\pi)$ be the corresponding matrices with the arbitrary scalar $k$ set to $0$. Then, we have the following derivation, where $k\mathbf{1}\sigma^\top$ and $kw_iw_j\mathbf{1}\pi_j^\top$ are eliminated by $(I-\mathbf{1}\sigma^\top)$ and $(I-\mathbf{1}\pi_j^\top)$, and $(I-\mathbf{1}\sigma^\top)$ and $(I-\mathbf{1}\pi_j^\top)$ are absorbed into $\bar{J}_{0ij}(\sigma)$ and $J_{0ij}(\pi)$ by idempotence.
  \begin{align*}
    \sigma_i\circ\bar{J}_i(\sigma)\widetilde{d\theta} & =\sigma_i\circ\bar{J}_i(\sigma)\left(I-\mathbf{1}\sigma^\top\right)d\theta=\sigma_i\circ\textstyle\sum_j \bar{J}_{0ij}(\sigma)d\theta_j \\
    \pi_i\circ J_i(\pi)\widetilde{d\theta}            & =\pi_i\circ\textstyle\sum_j J_{ij}(\pi)\left(I-\mathbf{1}\pi_j^\top\right)d\theta_j=\pi_i\circ\textstyle\sum_j J_{0ij}(\pi)d\theta_j
  \end{align*}

  Then, by $\sigma_i\circ\bar{J}_i(\sigma)\widetilde{d\theta}=\pi_i\circ J_i(\pi)\widetilde{d\theta}$ for every $d\theta$, we have $w_i\bar{J}_{0ij}(\sigma)=J_{0ij}(\pi)$. From $w_i(k\mathbf{1}\sigma^\top)_{ij}=w_i(k\mathbf{1}\sigma_j^\top)=kw_iw_j\mathbf{1}\pi_j^\top$, we obtain $w_i\bar{J}_{ij}(\sigma)=J_{ij}(\pi)$.
\end{proof}

Theorem~\ref{thmequ_diff} establishes the equivalence between the differentials of $\tilde{\tau}(\pi)$ and $\tilde{\mu}(\sigma)$, which govern the differential properties of the fixed-point bundles. It also yields $w_i\bar{J}_{ij}(\sigma)=J_{ij}(\pi)$ between the two Jacobian-like matrices, which play a central role in the predictor-corrector algorithm. The derivative of $\bar{F}(\sigma)$ is given by
\begin{equation}
  \frac{\partial\bar{F}_i(\sigma)}{\partial\sigma_j}=\frac{1}{w_iw_j}\left(I-\mathbf{1}\pi_i^\top\right)\frac{\partial F_i(\pi)}{\partial\pi_j}\left(I-\pi_j^\top\mathbf{1}\right)-\delta_{ij}\frac{1}{w_i}\left(\bar{F}_i(\sigma)\mathbf{1}^\top+\mathbf{1}\bar{F}_i(\sigma)^\top\right).
\end{equation}
Directly verifying $w_i\bar{J}_{ij}(\sigma)=J_{ij}(\pi)$ by symbolic derivation from equation~\eqref{equ_j} would be very complicated, but our numerical tests show that no significant difference can be found between ${\rm Diag}[w_iI]\bar{J}(\sigma)$ and $J(\pi)$ with $k=0$.

\section{Path-following on the fixed-point bundles}

Having established the equivalence between the two fixed-point bundle frameworks of ${\rm VI}(\bar{F},\Delta_{mn}^+)$ and ${\rm VI}(F,\Delta_n^m)$ via a fiber bundle isomorphism, we transform the path-following algorithm in Paper 1 so that it applies to ${\rm VI}(F,\Delta_n^m)$. Then the property of the transformed path-following algorithm provides a transition from ${\rm VI}(\bar{F},\Delta_{mn}^+)$ to an actual smooth VI on a closed simplex, so that the theoretical guarantee in Paper 1 transfers.

\subsection{Transforming the predictor-corrector framework}

Paper 1 adopts the predictor-corrector framework for path-following on the fixed-point bundle. This framework alternates between predictor steps and corrector steps: a predictor step updates along the tangent direction of the fixed-point bundle, given by the differential equations of $\bar{G}(\sigma,\mu)=0$ or $G(\pi,\tau)=0$, while a corrector step updates onto the fixed-point bundle, given by the Newton equations of $\bar{G}(\sigma,\mu)=0$ or $G(\pi,\tau)=0$, or by the gradients of ${\rm mcp}(\sigma,\mu)$ and ${\rm mcp}(\pi,\tau)$. The predictor and corrector for $\bar{G}(\sigma,\mu)=0$ are obtained by substituting $\bar{F}(\sigma)$ into Paper 1. In this subsection, we use the equivalence established above to obtain the predictor and corrector for $G(\pi,\tau)=0$.

\noindent\textbf{Differential equations and Newton equations of $\bar{G}(\sigma,\mu)=0$ and $G(\pi,\tau)=0$.}
In Paper 1, the differential equation and Newton equation of $\bar{G}(\sigma,\mu)=0$ are given by equations~\eqref{equ_diff_simplex} and~\eqref{equ_newton_simplex}.
\begin{equation}
  \label{equ_diff_simplex}
  \left(\bar{J}(\sigma)+(\mathbf{1}^\top\mu)I\right)\overline{d\theta}=\left(I-\mathbf{1}\sigma^\top\right)(d\mu/\sigma)
\end{equation}
\begin{equation}
  \label{equ_newton_simplex}
  \begin{aligned}
    \left(\bar{J}(\sigma)+(\mathbf{1}^\top\mu)I\right)\overline{d\theta}=-\bar{G}(\sigma,\mu)/\sigma
  \end{aligned}
\end{equation}

\begin{proposition}
  \label{thm_diff_newton}
  Let $(\sigma,\mu)=\Phi(\pi,\tau)$ and $\phi^{-1}_w(\sigma)={\rm norm}([\mathbf{1}^\top\mu_i])$. In both differential equation~\eqref{equ_diff_simplex} and Newton equation~\eqref{equ_newton_simplex}, we have $\pi_i^\top\overline{d\theta}_i=0$ for every index $i$. Consequently, $\overline{d\theta}=\widetilde{d\theta}$ and $dw=0$.
\end{proposition}
\begin{proof}
  Differential equation~\eqref{equ_diff_simplex} and Newton equation~\eqref{equ_newton_simplex} can be derived as follows. We first multiply both sides by ${\rm Diag}[w_i I]$ in the first lines, and then substitute $w_i\bar{J}_{ij}(\sigma)=J_{ij}(\pi)$ and $w={\rm norm}([\mathbf{1}^\top\tau_i])$ in the second lines.
  \begin{align*}
    {\rm Diag}\left[w_i I\right]\left(\bar{J}(\sigma)+(\mathbf{1}^\top\mu)I\right)\overline{d\theta} & ={\rm Diag}\left[w_i I\right]\left(I-\mathbf{1}\sigma^\top\right)(d\mu/\sigma) \\
    \left(J(\pi)+{\rm Diag}\left[(\mathbf{1}^\top\tau_i)I\right]\right)\overline{d\theta}            & =\left[\left(I-\mathbf{1}\pi_i^\top\right)(d\tau_i/\pi_i)\right]
  \end{align*}
  \begin{align*}
    {\rm Diag}\left[w_i I\right]\left(\bar{J}(\sigma)+(\mathbf{1}^\top\mu)I\right)\overline{d\theta} & =-{\rm Diag}\left[w_i I\right]\left(\bar{G}(\sigma,\mu)/\sigma\right) \\
    \left(J(\pi)+{\rm Diag}\left[(\mathbf{1}^\top\tau_i)I\right]\right)\overline{d\theta}            & =-\left[G_i(\pi,\tau)/\pi_i\right]
  \end{align*}

  Then, multiplying the $i$-th row blocks on both sides by $\pi_i^\top$ and using $\pi_i^\top J_i(\pi)=0$, we obtain $\pi_i^\top\overline{d\theta}_i=0$. It follows from Lemma~\ref{thm_dtheta}~(ii) that $\overline{d\theta}=\widetilde{d\theta}$ and $dw=0$.
\end{proof}

Proposition~\ref{thmequ_chara} already shows that $\bar{G}(\sigma,\mu)=0$ and $G(\pi,\tau)=0$ are in fact the same function of the same variable with the same value. Then, from the proof of Proposition~\ref{thm_diff_newton} and $\overline{d\theta}=\widetilde{d\theta}$, it follows immediately that the differential equation and Newton equation of $G(\pi,\tau)=0$ are given by equations~\eqref{equ_diff} and~\eqref{equ_newton}.
\begin{equation}
  \label{equ_diff}
  \left(J(\pi)+{\rm Diag}\left[(\mathbf{1}^\top\tau_i)I\right]\right)\widetilde{d\theta}=\left[\left(I-\mathbf{1}\pi_i^\top\right)(d\tau_i/\pi_i)\right]
\end{equation}
\begin{equation}
  \label{equ_newton}
  \left(J(\pi)+{\rm Diag}\left[(\mathbf{1}^\top\tau_i)I\right]\right)\widetilde{d\theta}=\left[\left(I-\mathbf{1}\pi_i^\top\right)(F_i(\pi)-\tau_i/\pi_i)\right]
\end{equation}

These equations also indicate that $(\pi,\tau)\in E$ is a singular point if and only if $\sum_i\mathbf{1}^\top\tau_i$ is an eigenvalue of $-[(1/w_i)J_{ij}(\pi)]$. In Paper 1, $(\sigma,\mu)\in\bar{E}$ is a singular point if and only if $\mathbf{1}^\top\mu$ is an eigenvalue of $-\bar{J}(\sigma)$. This is consistent with the homeomorphism between $(\sigma,\mu)$ and $(\pi,\tau)$.

\noindent\textbf{Gradients of ${\rm mcp}(\sigma,\mu)$ and ${\rm mcp}(\pi,\tau)$.}
Recall that ${\rm mcp}(\sigma,\mu)$ is a variant of the original MCP characterizing ${\rm VI}(\bar{F},\Delta_{mn}^+)$ in Paper 1. Denote the original MCP by $\overline{\rm mcp}(\sigma,\mu)$. In Paper 1, the gradient of $\overline{\rm mcp}(\sigma,\mu)$ is derived as the first equality in the following equation.
\begin{align*}
  d\overline{\rm mcp}(\sigma,\mu)= & \left(\sigma-\hat{\sigma}\right)^\top \left(\bar{J}(\sigma)+\left(\bar{r}-\left(I-\mathbf{1}\sigma^\top\right)\bar{F}(\sigma)\right)\circ I\right)\overline{d\theta}=d{\rm mcp}(\sigma,\mu)
\end{align*}

In the first equality, $(\sigma-\hat{\sigma})^\top\mathbf{1}d\bar{v}=0$ is used to eliminate $d\bar{v}$ in the derivation in Paper 1. From $d\overline{\rm mcp}(\sigma,\mu)$ to $d{\rm mcp}(\sigma,\mu)$, the elimination of $d\bar{v}_i$ remains as shown in the following equation. Thus we obtain the second equality in the equation above, which indicates that $d{\rm mcp}(\sigma,\mu)$ yields the same expression.
\begin{align*}
  (\sigma-\mu/\bar{r})^\top [d\bar{v}_i\mathbf{1}]= & \sum_i(\sigma_i-\hat{\sigma}_i)^\top\mathbf{1}d\bar{v}_i=\sum_i(w_i\pi_i-w_i\hat{\pi}_i)^\top \mathbf{1}d\bar{v}_i=0
\end{align*}

Then we have the following derivation from that expression. The first line uses ${\rm mcp}(\pi,\tau)={\rm mcp}(\sigma,\mu)$ from Proposition~\ref{thmequ_chara}. The second line replaces $\overline{d\theta}=(I-\mathbf{1}\sigma^\top)d\theta$ by $(I-\mathbf{1}\sigma^\top)\widetilde{d\theta}$ and uses $\sigma^\top\widetilde{d\theta}=0$ from Lemma~\ref{thm_dtheta}~(i); the substitution is valid because $d\theta$ and $\widetilde{d\theta}$ lead to the same $d\pi$ in equation~\eqref{equ_differential}, since $I-\mathbf{1}\pi_i^\top$ is idempotent, which in turn leads to the same $d{\rm mcp}(\pi,\tau)$. The third line uses $\bar{r}_i=r_i/w_i$ and $\hat{\sigma}_i=w_i\hat{\pi}_i$ from Proposition~\ref{thmequ_chara}~(iii), $w_i\bar{J}_{ij}(\sigma)=J_{ij}(\pi)$ from Theorem~\ref{thmequ_diff}, and $\sigma^\top\bar{F}(\sigma)=0$ and $w_i\bar{F}_i(\sigma)=F_i(\pi)-(\pi_i^\top F_i(\pi))\mathbf{1}$. The fourth line uses $r_i=F_i(\pi)+v_i\mathbf{1}$ and $\hat{\pi}_i^\top F_i(\pi)+v_i=\mathbf{1}^\top\tau_i$.
\begin{align*}
    & d{\rm mcp}(\pi,\tau)=d{\rm mcp}(\sigma,\mu)                                                                                                                                         \\
  = & \left(\sigma-\hat{\sigma}\right)^\top \left(\bar{J}(\sigma)+\left(\bar{r}-\left(I-\mathbf{1}\sigma^\top\right)\bar{F}(\sigma)\right)\circ I\right)\widetilde{d\theta}               \\
  = & \sum_{i,j}\left(\pi_i-\hat{\pi}_i\right)^\top\left(J_{ij}(\pi)+\delta_{ij}\left(r_i-\left(F_i(\pi)-(\pi_i^\top F_i(\pi))\mathbf{1}\right)\right)\circ I\right)\widetilde{d\theta}_j \\
  = & \sum_{i,j}\left(\pi_i-\hat{\pi}_i\right)^\top\left(J_{ij}(\pi)+\delta_{ij}\left(\mathbf{1}^\top\tau_i+(\pi_i-\hat{\pi}_i)^\top F_i(\pi)\right)I\right)\widetilde{d\theta}_j
\end{align*}

Ignoring the higher-order term $((\pi_i-\hat{\pi}_i)^\top F_i(\pi))(\pi_i-\hat{\pi}_i)$, we obtain the inexact gradient of ${\rm mcp}(\pi,\tau)$ in equation~\eqref{equ_grad}.
\begin{equation}
  \label{equ_grad}
  \nabla{\rm mcp}(\pi,\tau)=\left(J(\pi)+{\rm Diag}\left[(\mathbf{1}^\top\tau_i)I\right]\right)^\top\left[\pi_i-\hat{\pi}_i\right]
\end{equation}

\noindent\textbf{Regularized Newton equations.}
In Paper 1, the standard Newton corrector is used to prove the convergence rate of the overall predictor-corrector path-following algorithm, and two regularized Newton correctors are provided for practical use. The regularized Newton equations make the coefficient matrix positive semidefinite and control its condition number through a regularization parameter, so that the iteration is stabilized and the update steps do not blow up near singularities. Here we derive the two corresponding regularized Newton correctors for ${\rm VI}(F,\Delta_n^m)$. We write $J_G:=J(\pi)+{\rm Diag}[(\mathbf{1}^\top\tau_i)I]$ as a shorthand notation.

Multiplying both sides of the standard Newton equation~\eqref{equ_newton} by $J_G$ and adding a regularization parameter $\delta_H$ to the coefficient, we obtain the regularized Newton equation~\eqref{equ_regu_kkt}.
\begin{equation}
  \label{equ_regu_kkt}
  \begin{aligned}
    \left(J_G^\top J_G+\delta_H I\right)d\theta= & -J_G^\top\left[\left(I-\mathbf{1}\pi_i^\top\right)(F_i(\pi)-\tau_i/\pi_i)\right] \\
    \widetilde{d\theta}_i=                       & \left(I-\mathbf{1}\pi_i^\top\right)d\theta_i                                     \\
  \end{aligned}
\end{equation}

In Paper 1, the right-hand side of the standard Newton equation~\eqref{equ_newton} can be transformed into $\sigma-\hat{\sigma}$, as the first line of the following equation shows, where $v_{gb}\mathbf{1}$ is changed to $[v_{gb,i}\mathbf{1}]$ to reflect the change from the original Brouwer function to the index-wise variant $(\hat{\sigma},\bar{r},\bar{v})=\bar{M}(\sigma,\mu)$. Given $(\sigma,\mu)=\Phi(\pi,\tau)$ and $\phi^{-1}_w(\sigma)={\rm norm}([\mathbf{1}^\top\mu_i])$, this transformation can be carried further as follows.
\begin{align*}
  \frac{\sigma}{\bar{r}}\circ\left(\bar{J}(\sigma)+(\mathbf{1}^\top\mu)I\right)\left(\overline{d\theta}+[v_{gb,i}\mathbf{1}]\right)                                 & =-\left(\sigma-\hat{\sigma}\right)     \\
  \frac{w_i\pi_i}{r_i/w_i}\circ\sum_j\left((1/w_i)J_{ij}(\pi)+\delta_{ij}(1/w_i)(\mathbf{1}^\top\tau_i)I\right)\left(\overline{d\theta}_i+v_{gb,i}\mathbf{1}\right) & =-\left(w_i\pi_i-w_i\hat{\pi}_i\right) \\
  \frac{\pi_i}{r_i}\circ\sum_j\left(J_{ij}(\pi)+\delta_{ij}(\mathbf{1}^\top\tau_i)I\right)\left(\overline{d\theta}_i+v_{gb,i}\mathbf{1}\right)                      & =-\left(\pi_i-\hat{\pi}_i\right)
\end{align*}

Then, by multiplying both sides by $J_G$ and adding a regularization parameter $\delta_H$ to the coefficient, we obtain the regularized Newton equation~\eqref{equ_regu_barr}.
\begin{equation}
  \label{equ_regu_barr}
  \begin{aligned}
    \left(J_G^\top\circ\left[\frac{\pi_i}{r_i}\right]\circ J_G+\delta_H I\right)d\theta= & -J_G^\top\left[\pi_i-\hat{\pi}_i\right]      \\
    \widetilde{d\theta}_i=                                                               & \left(I-\mathbf{1}\pi_i^\top\right)d\theta_i \\
  \end{aligned}
\end{equation}

\subsection{Transferring the convergence results}

So far, we have completed the fixed-point bundle framework of ${\rm VI}(\bar{F},\Delta_{mn}^+)$ and the composite fixed-point bundle framework of ${\rm VI}(F,\Delta_n^m)$. We can now address the final transition from ${\rm VI}(\bar{F},\Delta_{mn}^+)$ to an actual smooth VI on a closed simplex, so that the theoretical guarantees in Paper 1 apply.

As in Paper 1, the path-following iteration consists of four parts: (i) reducing $\mu>0$ or $\tau>0$ toward $0$; (ii) avoiding singular points along the fibers by adding $\beta\sigma$ or $\beta{\rm norm}([\mathbf{1}^\top\tau_i])\odot\pi$ such that $\mathbf{1}^\top\mu+\beta$ avoids the eigenvalues of $-\bar{J}(\sigma)$, or $\sum_i\mathbf{1}^\top\tau_i+\beta$ avoids the eigenvalues of $-{\rm Diag}[(1/w_i)I]J(\pi)$; (iii) a predictor step along the tangent direction of $\bar{G}(\sigma,\mu)=0$ or $G(\pi,\tau)=0$; (iv) a corrector step toward $\bar{G}(\sigma,\mu)=0$ or $G(\pi,\tau)=0$. First, we establish the equivalence between the path-following algorithms for ${\rm VI}(\bar{F},\Delta_{mn}^+)$ and ${\rm VI}(F,\Delta_n^m)$ with standard Newton correctors.

\begin{theorem}
  \label{thmequ_update}
  Let the starting points be $(\sigma_{\rm init},v_{\rm init}\sigma_{\rm init})$ and $(\pi_{\rm init},v_{\rm init}w_{\rm init}\odot\pi_{\rm init})$ such that $\sigma_{\rm init}=\phi(\pi_{\rm init},w_{\rm init})$. Let the iteration consist of the following three kinds of updates.
  \begin{enumerate}
    \item $\mu'=(1-\eta)\mu$ and $\tau'=(1-\eta)\tau$.
    \item $\hat{\mu}=\mu+\beta\sigma$ and $\hat{\tau}=\tau+\beta{\rm norm}([\mathbf{1}^\top\tau_i])\odot\pi$.
    \item $\sigma'_i=\phi^{-1}_{w,i}(\sigma){\rm softmax}(\ln\phi^{-1}_{\pi,i}(\sigma)+\overline{d\theta}_i)$ and $\pi'_i={\rm softmax}(\ln\pi_i+\widetilde{d\theta}_i)$ with $\overline{d\theta}$ and $\widetilde{d\theta}$ subject to the differential equations or the Newton equations of $\bar{G}(\sigma,\mu)=0$ and $G(\pi,\tau)=0$.
  \end{enumerate}
  Then, $(\sigma,\mu)=\Phi(\pi,\tau)$ and $w_{\rm init}=\phi^{-1}_w(\sigma)={\rm norm}([\mathbf{1}^\top\mu_i])$ hold constantly during the iteration.
\end{theorem}
\begin{proof}
  The first equality $(\sigma,\mu)=\Phi(\pi,\tau)$ holds if and only if $\sigma=\phi(\pi,w)$, $\mu={\rm reshape}(\tau)$, and $w={\rm norm}([\mathbf{1}^\top\mu_i])$. Given the first equality, the second equality holds if and only if $w=w_{\rm init}$. Thus, we use these four equalities to verify that the two conditions are preserved during the iteration.

  At the starting points $(\sigma_{\rm init},v_{\rm init}\sigma_{\rm init})$ and $(\pi_{\rm init},v_{\rm init}w_{\rm init}\odot\pi_{\rm init})$, from $\sigma_{\rm init}=\phi(\pi_{\rm init},w_{\rm init})$, we can verify $\sigma=\phi(\pi,w)$, $\mu={\rm reshape}(\tau)$, $\phi^{-1}_w(\sigma)=w_{\rm init}$, and ${\rm norm}([\mathbf{1}^\top\mu_i])=w_{\rm init}$. Thus, the conditions $(\sigma,\mu)=\Phi(\pi,\tau)$ and $w_{\rm init}=\phi^{-1}_w(\sigma)={\rm norm}([\mathbf{1}^\top\mu_i])$ hold at the starting points. Next, it remains to prove that given that the two conditions hold, the updates preserve them.

  Update (i): Since $\sigma$ and $\pi$ remain unchanged, we have $\sigma=\phi(\pi,w)$ and $w=w_{\rm init}$.For $\mu'$, we have ${\rm norm}([\mathbf{1}^\top\mu'_i])={\rm norm}([\mathbf{1}^\top\mu_i])$, so $w={\rm norm}([\mathbf{1}^\top\mu'_i])$ given $w={\rm norm}([\mathbf{1}^\top\mu_i])$. For $\mu'$ and $\tau'$, we have $\mu'={\rm reshape}(\tau')$ given $\mu={\rm reshape}(\tau)$. Thus, the two conditions are preserved.

  Update (ii): Since $\sigma$ and $\pi$ remain unchanged, we have $\sigma=\phi(\pi,w)$ and $w=w_{\rm init}$. For $\hat{\mu}$, we have ${\rm norm}([\mathbf{1}^\top\hat{\mu}_i])={\rm norm}([\mathbf{1}^\top\mu_i+\beta w_{{\rm init},i}])=w_{\rm init}$ given $\mathbf{1}^\top\sigma_i=w_i$ and $w={\rm norm}([\mathbf{1}^\top\mu_i])$. For $\hat{\mu}$ and $\hat{\tau}$, we have $\sigma={\rm reshape}({\rm norm}([\mathbf{1}^\top\tau_i])\odot\pi)$ given $\mu={\rm reshape}(\tau)$ and ${\rm norm}([\mathbf{1}^\top\mu_i])=w$, then we have $\hat{\mu}={\rm reshape}(\hat{\tau})$. Thus, the two conditions are preserved.

  Update (iii): The update $\sigma'_i=\phi^{-1}_{w,i}(\sigma){\rm softmax}(\ln\phi^{-1}_{\pi,i}(\sigma)+\overline{d\theta}_i)$ preserves $\phi^{-1}_w(\sigma)=w$ such that $\phi^{-1}_w(\sigma')=w$ as well, and then $w=w_{\rm init}$ is also preserved. Since $\mu$ and $\tau$ remain unchanged, we have $\mu={\rm reshape}(\tau)$ and $w={\rm norm}([\mathbf{1}^\top\mu_i])$. Proposition~\ref{thm_diff_newton} shows that $\overline{d\theta}=\widetilde{d\theta}$ for either the differential equations or the Newton equations, then ${\rm softmax}(\ln\phi^{-1}_{\pi,i}(\sigma)+\overline{d\theta}_i)=\pi'_i$ given $\sigma=\phi(\pi,w)$, which indicates $\sigma'=\phi(\pi',w)$. Thus, the two conditions are preserved.

  Therefore, the conditions hold at the starting points, and the updates preserve them, so they hold constantly during the iteration.
\end{proof}

In Theorem~\ref{thmequ_update}~(iii), $\sigma$ is updated by $\sigma'_i=\phi^{-1}_{w,i}(\sigma){\rm softmax}(\ln\phi^{-1}_{\pi,i}(\sigma)+\overline{d\theta}_i)$, instead of the Paper 1 version $\sigma'={\rm softmax}(\ln\sigma+\overline{d\theta})$. In fact, the former is a second-order correction of the latter, enforcing $w$ as an exact constant on top of the $dw=0$ implied by Proposition~\ref{thm_diff_newton}. The difference between the two updates for $\sigma$ is derived as follows, where the second line uses $d\pi_i=\pi_i\circ\widetilde{d\theta}_i$ from equation~\eqref{equ_differential} and $\overline{d\theta}=\widetilde{d\theta}$ from Proposition~\ref{thm_diff_newton}.
\begin{align*}
    & {\rm softmax}_i(\ln\sigma+\overline{d\theta})-\phi^{-1}_{w,i}(\sigma){\rm softmax}(\ln\phi^{-1}_{\pi,i}(\sigma)+\overline{d\theta}_i)                                                                                    \\
  = & \sigma_i+\sigma_i\circ\overline{d\theta}_i+o(\overline{d\theta}_i)-\phi^{-1}_{w,i}(\sigma)(\phi^{-1}_{\pi,i}(\sigma)+\phi^{-1}_{\pi,i}(\sigma)\circ\overline{d\theta}_i+o(\overline{d\theta}_i))=o(\overline{d\theta}_i)
\end{align*}
A second-order correction $o(\overline{d\theta}_i)$ at each update does not change the convergence rate of the original predictor-corrector framework \cite{allgower2003introduction}.

Theorem~\ref{thmequ_update} shows that the predictor-corrector path-following algorithms on $\bar{E}\to\Delta_{mn}^+$ and $E\to\Delta_m^\circ\times\Delta_n^m\to\Delta_n^m$ generate sequences of $(\sigma,\mu)$ and $(\pi,\tau)$ that correspond exactly via the homeomorphism $\Phi$, and that $w$ remains constant throughout the iteration. This can be used for the final transition of constructing a smooth ${\rm VI}(\tilde{F},\Delta_{mn})$ on the closed simplex such that all results in Paper 1 apply. Specifically, we want a map $\tilde{F}$ that is real-analytic on the closed simplex $\Delta_{mn}$, such that $\tilde{F}$ and $\bar{F}$ have the same value on the subset $W(w_{\rm init})=\{\sigma\in\Delta_{mn}|\phi^{-1}_{w,i}(\sigma)=w_{\rm init}\}$, which is given by equation~\eqref{equ_smooth_vi}.
\begin{equation}
  \label{equ_smooth_vi}
  \begin{aligned}
    \tilde{F}_i(\sigma)      = & \frac{1}{w_{{\rm init},i}}\left(F_i(\tilde{\pi})-(\tilde{\pi}_i^\top F_i(\tilde{\pi}))\mathbf{1}\right)                                         \\
    {\rm s.t.~}                & \tilde{\pi}_i(\sigma)=\left(1-\frac{\phi^{-1}_{w,i}(\sigma)}{\tilde{w}_i(\sigma)}\right)\pi_{{\rm init},i}+\frac{\sigma_i}{\tilde{w}_i(\sigma)} \\
                               & \tilde{w}(\sigma)        =\phi^{-1}_w(\sigma)+(w_{\rm init}-\phi^{-1}_w(\sigma))^2                                                              \\
  \end{aligned}
\end{equation}

For every $\sigma\in\Delta_{mn}$, we have $\tilde{\pi}(\sigma)\in\Delta_n^m$, because $\mathbf{1}^\top\tilde{\pi}_i(\sigma)=1$ follows from $\mathbf{1}^\top\sigma_i=\phi^{-1}_{w,i}(\sigma)$ and $\mathbf{1}^\top\pi_{{\rm init},i}=1$, and $\tilde{\pi}_i(\sigma)\geq0$ follows from $\tilde{w}(\sigma)\geq\phi^{-1}_w(\sigma)$. Since, in addition, $F$ is real-analytic on $\Delta_n^m$ and the denominators $w_{\rm init}>0$ and $\tilde{w}(\sigma)>0$, we obtain that $\tilde{F}(\sigma)$ is real-analytic on the simplex $\Delta_{mn}$. Furthermore, when $\phi^{-1}_w(\sigma)=w_{\rm init}$, we have $\tilde{w}(\sigma)=w_{\rm init}$ and $\tilde{\pi}(\sigma)=\phi^{-1}_\pi(\sigma)$, so that $\tilde{F}(\sigma)=\bar{F}(\sigma)$. Thus, $\tilde{F}$ and $\bar{F}$ have the same value on the subset $W(w_{\rm init})$.

Besides their function values, we also need to compare the derivatives of $\tilde{F}$ and $\bar{F}$.
\begin{proposition}
  \label{thm_final}
  On $W(w_{\rm init})$, the differential equation~\eqref{equ_diff_simplex} or Newton equation~\eqref{equ_newton_simplex} remains the same when the operator $\bar{F}$ is replaced by $\tilde{F}$.
\end{proposition}
\begin{proof}
  Proposition~\ref{thm_diff_newton} shows that $dw=0$ in differential equation~\eqref{equ_diff_simplex} and Newton equation~\eqref{equ_newton_simplex}. For the differentials $d\tilde{F}(\sigma)$ and $d\bar{F}(\sigma)$, if $\sigma\in W(w_{\rm init})$ and $dw=0$, we have $d\tilde{F}(\sigma)=d\bar{F}(\sigma)$ since $\tilde{F}(\sigma)$ and $\bar{F}(\sigma)$ have the same value on $W(w_{\rm init})$. Then, expanding $d\tilde{F}(\sigma)$ and $d\bar{F}(\sigma)$, we have
  $$\frac{\partial\tilde{F}(\sigma)}{\partial\sigma}\circ\sigma\circ\overline{d\theta}=\frac{\partial\bar{F}(\sigma)}{\partial\sigma}\circ\sigma\circ\overline{d\theta}.$$
  Note that the differential equations and Newton equations depend on the operator $\bar{F}$ only through its value $\bar{F}(\sigma)$ and its differential $(\partial\tilde{F}(\sigma)/\partial\sigma)\circ\sigma\circ\overline{d\theta}$. Thus, we obtain that on $W(w_{\rm init})$, the differential equations or Newton equations are the same for operators $\bar{F}$ and $\tilde{F}$.
\end{proof}

Theorem~\ref{thmequ_update} shows that the update of the predictor-corrector path-following algorithm keeps $\sigma$ on $W(w_{\rm init})$ at all times. Proposition~\ref{thm_final} shows that on $W(w_{\rm init})$ the differential equations and Newton equations are the same for the operators $\bar{F}$ and $\tilde{F}$. Therefore, the predictor-corrector path-following algorithms applied to ${\rm VI}(\bar{F},\Delta_{mn}^+)$ and ${\rm VI}(\tilde{F},\Delta_{mn})$ produce exactly the same iteration sequence.

The problem ${\rm VI}(\tilde{F},\Delta_{mn})$ is a real-analytic VI on the closed simplex, so all the results in Paper 1 apply to it. The theoretical guarantees of the path-following on the fixed-point bundle then transfer from ${\rm VI}(\tilde{F},\Delta_{mn})$ to ${\rm VI}(\bar{F},\Delta_{mn}^+)$, and then to ${\rm VI}(F,\Delta_n^m)$.

\subsection{The complete path-following algorithm}

As in Paper 1, in practice we exploit the geometric freedom of the fixed-point bundle without necessarily following the path constructed in the proofs, so as to improve the robustness of the algorithm. Now that we have shown that $w$ can remain unchanged during the iteration, we can add a mechanism to modify $w$ arbitrarily in singularity avoidance. Specifically, we reallocate $\mathbf{1}^\top\tau_i$ to change $w={\rm norm}([\mathbf{1}^\top\tau_i])$.

As in Paper 1, $\tau$ must satisfy $\tau>0$ to ensure path-following that leads to solutions of ${\rm VI}(F,\Delta_n^m)$. Since $\check{\tau}(\pi)$ is a section of the fixed-point bundle, for any $(\pi,\tau)\in E$ we have $\tau_i=\check{\tau}_i(\pi)+k_i\pi_i$. Because $\min_a\check{\tau}_{i,a}(\pi)=0$, we have $\tau>0$ if and only if $\mathbf{1}^\top\tau_i>\mathbf{1}^\top\check{\tau}_i(\pi)$. Thus $\hat{\tau}>0$ and $\sum_i\mathbf{1}^\top\hat{\tau}_i=\sum_i\mathbf{1}^\top\tau_i+\beta$ hold if and only if there exists $\tilde{w}_i\in\Delta_m^\circ$ such that $\hat{\tau}$ satisfies
\begin{equation}
  \label{equ_singu_avoid}
  \mathbf{1}^\top\hat{\tau}_i=\tilde{w}_i\left(\sum_i\mathbf{1}^\top\tau_i+\beta-\sum_i\mathbf{1}^\top\check{\tau}_i(\pi)\right)+\mathbf{1}^\top\check{\tau}_i(\pi).
\end{equation}

In singularity avoidance, we select $\tilde{w}\in\Delta_m^\circ$ and $\beta$ such that $\sum_i\mathbf{1}^\top\tau_i+\beta$ avoids the eigenvalues of $-{\rm Diag}[(1/w_i)I]J(\pi)$, where $w_i=\mathbf{1}^\top\hat{\tau}_i/(\sum_j\mathbf{1}^\top\hat{\tau}_j)$ is given by equation~\eqref{equ_singu_avoid}. The value of $\hat{\tau}$ after singularity avoidance is then $\hat{\tau}=\tilde{\tau}(\pi)+(\sum_i\mathbf{1}^\top\tau_i+\beta)w\odot\pi$. Our experimental results are based on this singularity avoidance mechanism with $\tilde{w}\in\Delta_m^\circ$ chosen randomly, which performs better than keeping $w$ fixed.

The path-following algorithm for ${\rm VI}(F,\Delta_n^m)$ is given in Algorithm~\ref{algo}. Apart from its singularity avoidance step, it corresponds to Algorithm 1 in Paper 1 via the equivalence established in this paper.

\begin{algorithm}
  \caption{Path-following on the composite fixed-point bundle}
  \label{algo}
  \begin{algorithmic}[1]
    \Require A smooth map $F:\Delta_n^m\to\mathbb{R}^{m\times n}$ and its derivative $[\partial F_i/\partial\pi_j]$, a designated starting point $\pi_{\rm init}\in\Delta_n^m$, and a desired precision $\epsilon>0$
    \State Set $(\pi_0,\tau_0)=(\pi_{\rm init},v_{\rm init}w_{\rm init}\odot\pi_{\rm init})$ for a $w_{\rm init}\in\Delta_m^\circ$ and a sufficiently large $v_{\rm init}$
    \Repeat
    \Repeat
    \State Compute $(\hat{\pi}_k,r_k,v_k)=M(\pi_k,\tau_t)$ by solving bisection problem \eqref{equ_bisection}
    \State Construct the matrix $J(\pi_k)$ in equation~\eqref{equ_j}
    \State Solve regularized Newton equation~\eqref{equ_regu_kkt} or~\eqref{equ_regu_barr} for $\widetilde{d\theta}_k$ with $\delta_H=\lVert G(\pi_k,\tau_t)\rVert/(mn)$
    \State Update $\pi_{k+1,i}={\rm softmax}(\ln\pi_{k,i}+\widetilde{d\theta}_{k,i})$
    \Until{$\lVert G(\pi_k,\tau_t)\rVert\leq\epsilon/(mn)$}
    \State Set $\hat{\tau}_{t,i}=\tilde{\tau}_i(\pi_k)+(\mathbf{1}^\top\hat{\tau}_{t,i})\pi_k$ with $\mathbf{1}^\top\hat{\tau}_{t,i}$ given by equation~\eqref{equ_singu_avoid}, such that $\sum_i\mathbf{1}^\top\tau_{t,i}+\beta_t$ avoids the nonzero eigenvalues of $-{\rm Diag}[(1/w_i)I]J(\pi_k)$
    \State Update $\tau_{t+1}=(\hat{\tau}_t-\eta_t u_t).{\rm clip}(\min=\epsilon/(mn))$ with a small $\eta_t$ and a direction $u_t=\hat{\tau}_t$
    \State Solve differential equation~\eqref{equ_diff} for $\widetilde{d\theta}_t$ with $d\tau_t=\tau_{t+1}-\hat{\tau}_t$
    \State Set $\pi_{0,i}={\rm softmax}(\ln\pi_{k,i}+\widetilde{d\theta}_{t,i})$
    \Until{$\textstyle\sum_i\mathbf{1}^\top\check{\tau}_i(\pi_k)\leq\epsilon$}\\
    \Return $\pi_k$ as an approximate solution of ${\rm VI}(F,\Delta_n^m)$
  \end{algorithmic}
\end{algorithm}

The corrector can be either regularized Newton equation~\eqref{equ_regu_kkt} or regularized Newton equation~\eqref{equ_regu_barr}. As in Paper 1, the former mixes the standard Newton equation~\eqref{equ_newton} with the inexact gradient $J_G^\top[G_i(\pi,\tau)/\pi_i]$ of $\lVert[G_i(\pi,\tau)/\pi_i]\rVert$, while the latter mixes the standard Newton equation~\eqref{equ_newton} with the inexact gradient $J_G^\top[\pi_i-\hat{\pi}_i]$ of ${\rm mcp}(\pi,\tau)$.

\section{Numerical experiments}

We test the algorithm on randomly generated instances of ${\rm VI}(F,\Delta_n^m)$, with $F$ represented by two different models: a neural network model and a payoff tensor model. The neural network model represents VI problems with a general real-analytic operator. In our experiments, we use tanh-activated neural networks with architecture $[mn,mn,mn]$ to represent $F$. The parameters of the neural networks are generated by Xavier normal initialization, which prevents the tanh activations from being saturated and hence decreases the complexity of $F$.

The payoff tensor model represents finite normal-form games. For an $m$-player $n$-action finite normal-form game, the payoff function is a multilinear function given by an $m$-th order payoff tensor of shape $(n,\dots,n)$. The tensor consists of $n^m$ elements, which grows exponentially with $m$. We use the classical low-rank CP (canonical polyadic) representation as a practical alternative to the full tensor representation \cite{kolda2009tensor}. Specifically, $F$ is represented by $F_i(\pi)=\nabla_{\pi_i}f_i(\pi)$, where
\begin{equation}
  f_i(\pi)=\sum_{r=1}^R\prod_{j=1}^m(U_{r,i,j}\pi_j)
\end{equation}
is the low-rank CP payoff function. Each $U_i$ is a third-order tensor of shape $(R,m,n)$, which reduces the total storage requirement from $n^m$ to $Rmn$. In our experiments, we use $R=mn$. The parameters of $U$ are generated from a normal distribution with mean $1$ and variance $1/m$, so that $\prod_{j=1}^mu_j$ has mean $1$ and variance close to $e-1$, which prevents overflow or underflow for very large $m$.

The starting point $\pi_{\rm init}$ is also randomly generated, and the required precision is $\epsilon=10^{-5}$. We solve the problems until ${\rm gap}_i(\pi)<\epsilon$ rather than ${\rm gap}(\pi)=\sum_i{\rm gap}_i(\pi)<\epsilon$. For dimensions, we fix the player number $m=8$ and vary the action number $n$ from $2$ to $32$, and fix the action number $n=8$ and vary the player number $m$ from $2$ to $32$. Since the algorithm operates in dimension $mn$, we also fix $mn=256$ and test $(m,n)$ ranging from $(2,128)$ to $(128,2)$. For each data point, we run 200 instances, totaling 5600 instances. We test the algorithm only with the regularized Newton corrector~\eqref{equ_regu_barr}, which is faster as reported in Paper 1. Table~\ref{table} reports the results.

\begin{table}[htbp]
  \centering
  \caption{Statistics (mean / median) solving ${\rm VI}(F,\Delta_n^m)$}
  \label{table}
  \begin{tabular}{lcccc}
    \toprule
    \multirow{2}{*}{}
            & \multicolumn{2}{c}{Neural network F}
            & \multicolumn{2}{c}{Payoff F}                                                                                                               \\
    \cmidrule(lr){2-3} \cmidrule(lr){4-5}
    $m,n$   & Iteration count                      & \makecell{Singularity\\avoidance times} & Iteration count & \makecell{Singularity\\avoidance times} \\
    \midrule
    $8,2$   & 369 / 350                            & 0.8 / 0.0                               & 360 / 353       & 0.2 / 0.0                               \\
    $8,4$   & 446 / 387                            & 3.3 / 0.0                               & 439 / 394       & 2.0 / 0.0                               \\
    $8,8$   & 618 / 517                            & 10.2 / 5.0                              & 548 / 442       & 5.3 / 0.0                               \\
    $8,16$  & 907 / 789                            & 23.4 / 18.5                             & 632 / 565       & 7.9 / 5.0                               \\
    $8,32$  & 2067 / 1344                          & 75.0 / 44.0                             & 823 / 702       & 14.7 / 10.0                             \\
    \midrule
    $2,8$   & 394 / 356                            & 2.0 / 0.0                               & 409 / 403       & 0.3 / 0.0                               \\
    $4,8$   & 435 / 375                            & 3.4 / 0.0                               & 454 / 412       & 2.0 / 0.0                               \\
    $8,8$   & 618 / 517                            & 10.2 / 5.0                              & 548 / 442       & 5.3 / 0.0                               \\
    $16,8$  & 1053 / 836                           & 27.2 / 18.0                             & 710 / 640       & 11.3 / 9.0                              \\
    $32,8$  & 2561 / 1823                          & 85.4 / 58.5                             & 1287 / 1133     & 31.0 / 25.0                             \\
    \midrule
    $2,128$ & 1203 / 825                           & 48.6 / 30.0                             & 532 / 477       & 2.5 / 0.0                               \\
    $4,64$  & 1616 / 1084                          & 62.1 / 37.5                             & 619 / 533       & 6.0 / 1.0                               \\
    $8,32$  & 2067 / 1344                          & 75.0 / 44.0                             & 823 / 702       & 14.7 / 10.0                             \\
    $16,16$ & 2128 / 1518                          & 70.8 / 44.5                             & 1054 / 929      & 23.3 / 19.0                             \\
    $32,8$  & 2561 / 1823                          & 85.4 / 58.5                             & 1287 / 1133     & 31.0 / 25.0                             \\
    $64,4$  & 2616 / 1790                          & 85.8 / 53.5                             & 1544 / 1265     & 40.9 / 31.0                             \\
    $128,2$ & 1142 / 955                           & 30.7 / 23.0                             & 942 / 754       & 20.1 / 12.5                             \\
    \bottomrule
  \end{tabular}
\end{table}

The algorithm succeeds on all 5600 instances. The iteration count increases faster when $m$ increases with $n$ fixed than when $n$ increases with $m$ fixed. When $mn$ is fixed, the peak iteration count occurs in the middle rather than at the endpoints $(2,128)$ or $(128,2)$, indicating that both $m$ and $n$ affect the problem difficulty for a fixed total dimension $mn$. In addition, solving ${\rm VI}(F,\Delta_n^m)$ with the payoff $F$ is significantly easier than with the neural network $F$, because neural networks produce greater nonlinearity and a more twisted singular manifold, causing much more singularity avoidance.

\section{Conclusion}

This paper generalizes the fixed-point bundle framework in Paper 1 for VIs with a simplex domain to VIs with a product-of-simplex domain. The fixed-point bundle $E\to\Delta_n^m\times\Delta_m^\circ\to\Delta_n^m$ of ${\rm VI}(F,\Delta_n^m)$ has a composite fiber bundle structure.

We construct ${\rm VI}(\bar{F},\Delta_{mn}^+)$ with simplex domain from ${\rm VI}(F,\Delta_n^m)$ that preserves the solutions. The central innovation of this paper is to establish the equivalence between their fixed-point bundle frameworks, based on $\bar{E}\to\Delta_{mn}^+$ and $E\to\Delta_n^m\times\Delta_m^\circ$, via the bundle isomorphism $(\Phi:E\to\bar{E},\phi:\Delta_n^m\times\Delta_m^\circ\to\Delta_{mn}^+)$.

The path-following algorithms on the fixed-point bundles keep $w$ constant. The restriction of the map $\bar{F}$ to the slice $w=w_{\rm init}$ can be extended to a real-analytic map $\tilde{F}$ on the closed simplex $\Delta_{mn}$, and then all the results in Paper 1 apply to ${\rm VI}(\tilde{F},\Delta_{mn})$. Thus, the path-following algorithm for ${\rm VI}(F,\Delta_n^m)$ has the same convergence guarantee as in Paper 1. Numerical experiments on 5600 randomly generated problems with dimensions ranging from 2-player 128-action to 128-player 2-action confirm universal success.

The singularity avoidance in the framework involves arbitrarily selecting a new $w$, and we currently adopt the most basic random selection method. Although simple and effective in experiments, computing $w$ from the state of the path-following process could further improve practical performance.

\bibliographystyle{unsrt}

\begin{thebibliography}{10}

\bibitem{nash1951non}
John Nash.
\newblock Non-cooperative games.
\newblock {\em Annals of mathematics}, 54(2):286--295, 1951.

\bibitem{kreps1990game}
David~M Kreps.
\newblock {\em Game theory and economic modelling}.
\newblock Oxford University Press, 1990.

\bibitem{han2012game}
Zhu Han.
\newblock {\em Game theory in wireless and communication networks: theory, models, and applications}.
\newblock Cambridge university press, 2012.

\bibitem{hazra2022applications}
Tanmoy Hazra and Kushal Anjaria.
\newblock Applications of game theory in deep learning: a survey.
\newblock {\em Multimedia Tools and Applications}, 81(6):8963--8994, 2022.

\bibitem{dual_gap}
Francisco Facchinei and Jong-Shi Pang.
\newblock {\em Finite-dimensional variational inequalities and complementarity problems}.
\newblock Springer, 2003.

\bibitem{facchinei2007generalized}
Francisco Facchinei, Andreas Fischer, and Veronica Piccialli.
\newblock On generalized nash games and variational inequalities.
\newblock {\em Operations Research Letters}, 35(2):159--164, 2007.

\bibitem{belgioioso2018projected}
Giuseppe Belgioioso and Sergio Grammatico.
\newblock Projected-gradient algorithms for generalized equilibrium seeking in aggregative games are preconditioned forward-backward methods.
\newblock {\em arXiv preprint arXiv:1803.10441}, 2018.

\bibitem{yi2018distributed}
Peng Yi and Lacra Pavel.
\newblock Distributed generalized nash equilibria computation of monotone games via double-layer preconditioned proximal-point algorithms.
\newblock {\em IEEE Transactions on Control of Network Systems}, 6(1):299--311, 2018.

\bibitem{dreves2013globalized}
Axel Dreves, Anna von Heusinger, Christian Kanzow, and Masao Fukushima.
\newblock A globalized newton method for the computation of normalized nash equilibria.
\newblock {\em Journal of Global Optimization}, 56(2):327--340, 2013.

\bibitem{dreves2011solution}
Axel Dreves, Francisco Facchinei, Christian Kanzow, and Simone Sagratella.
\newblock On the solution of the kkt conditions of generalized nash equilibrium problems.
\newblock {\em SIAM Journal on Optimization}, 21(3):1082--1108, 2011.

\bibitem{herings2010homotopy}
P~Jean-Jacques Herings and Ronald Peeters.
\newblock Homotopy methods to compute equilibria in game theory.
\newblock {\em Economic Theory}, 42(1):119--156, 2010.

\bibitem{yi2019operator}
Peng Yi and Lacra Pavel.
\newblock An operator splitting approach for distributed generalized nash equilibria computation.
\newblock {\em Automatica}, 102:111--121, 2019.

\bibitem{vinh2019inertial}
Nguyen~The Vinh and Le~Dung Muu.
\newblock Inertial extragradient algorithms for solving equilibrium problems.
\newblock {\em Acta Mathematica Vietnamica}, 44(3):639--663, 2019.

\bibitem{franci2022stochastic}
Barbara Franci and Sergio Grammatico.
\newblock Stochastic generalized nash equilibrium seeking under partial-decision information.
\newblock {\em Automatica}, 137:110101, 2022.

\bibitem{fp}
George~W Brown.
\newblock Iterative solution of games by fictitious play.
\newblock {\em Activity Analysis of Production and Allocation}, 13(1):374, 1951.

\bibitem{fsp}
Johannes Heinrich, Marc Lanctot, and David Silver.
\newblock Fictitious self-play in extensive-form games.
\newblock In Francis Bach and David Blei, editors, {\em Proceedings of the 32nd International Conference on Machine Learning}, volume~37, pages 805--813, Lille, France, 2015. PMLR.

\bibitem{noregret}
Sergiu Hart and Andreu Mas-Colell.
\newblock A simple adaptive procedure leading to correlated equilibrium.
\newblock {\em Econometrica}, 68(5):1127--1150, 2000.

\bibitem{cfr}
Martin Zinkevich, Michael Johanson, Michael Bowling, and Carmelo Piccione.
\newblock Regret minimization in games with incomplete information.
\newblock In J.~Platt, D.~Koller, Y.~Singer, and S.~Roweis, editors, {\em Advances in Neural Information Processing Systems}, volume~20. Curran Associates, Inc., 2007.

\bibitem{ppad}
Christos~H Papadimitriou.
\newblock On the complexity of the parity argument and other inefficient proofs of existence.
\newblock {\em Journal of Computer and system Sciences}, 48(3):498--532, 1994.

\bibitem{mine}
Hongbo Sun.
\newblock A path-following framework on fiber bundle for variational inequalities, 2026.

\bibitem{krantz2001function}
Steven~George Krantz.
\newblock {\em Function theory of several complex variables}, volume 340.
\newblock American Mathematical Soc., 2001.

\bibitem{analytic}
Walter Rudin.
\newblock {\em Real and complex analysis}.
\newblock McGraw-Hill Education, 1974.

\bibitem{normal_cone}
R~Tyrrell Rockafellar and Roger~JB Wets.
\newblock {\em Variational analysis}.
\newblock Springer, 1998.

\bibitem{allgower2003introduction}
Eugene~L Allgower and Kurt Georg.
\newblock {\em Introduction to numerical continuation methods}.
\newblock SIAM, 2003.

\bibitem{kolda2009tensor}
Tamara~G Kolda and Brett~W Bader.
\newblock Tensor decompositions and applications.
\newblock {\em SIAM review}, 51(3):455--500, 2009.

\end{thebibliography}

\end{document}